\documentclass[]{amsart}
\usepackage{latexsym,xy}
\xyoption{all}
\usepackage{graphicx,amscd,amssymb,latexsym,subfigure,verbatim,hyperref,epsfig,supertabular}

\newtheorem{defn0}{Definition}[section]
\newtheorem{prop0}[defn0]{Proposition}
\newtheorem{conj0}[defn0]{Conjecture}
\newtheorem{thm0}[defn0]{Theorem}
\newtheorem{lem0}[defn0]{Lemma}
\newtheorem{corollary0}[defn0]{Corollary}
\newtheorem{example0}[defn0]{Example}
\newtheorem{remark0}[defn0]{Remark}
\newtheorem{que0}[defn0]{Question}

\newenvironment{defn}{\begin{defn0}}{\end{defn0}}
\newenvironment{prop}{\begin{prop0}}{\end{prop0}}

\newenvironment{thm}{\begin{thm0}}{\end{thm0}}
\newenvironment{lem}{\begin{lem0}}{\end{lem0}}

\newenvironment{cor}{\begin{corollary0}}{\end{corollary0}}
\newenvironment{exm}{\begin{example0}\rm}{\end{example0}}
\newenvironment{remark}{\begin{remark0}\rm}{\end{remark0}}
\newcommand{\V}{{\bf V }}
\renewcommand{\P}{{\mathbb{P}}}

\newcommand{\FF}{{\mathbb{F}}}

\newcommand{\OPP}{{\mathcal{O}_{\mathbb{P}^1 \times \mathbb{P}^1}}}
\newcommand{\PP}{{\mathbb{P}^1 \times \mathbb{P}^1}}

\newcommand{\OPone}{{\mathcal{O}_{\mathbb{P}^1}}}
\newcommand{\spn}{{\mathrm{Span}}}

\newcommand{\DET}{{\mathrm{det}}}
\newcommand{\rank}{{\mathrm{rank}}}
\newcommand{\im}{{\mathrm{im}}}
\newcommand{\depth}{{\mathrm{depth}}}
\newcommand{\coker}{{\mathrm{coker}}}
\newcommand{\Ass}{{\mathrm{Ass}}}
\newcommand{\Sing}{{\mathrm{Sing}}}
\newcommand{\Pic}{{\mathrm{Pic}}}

\newcommand{\mm}{{\mathfrak{m}}}

\numberwithin{equation}{section}

\begin{document}
\title[Syzygies and singularities of tensor product surfaces]%
{Syzygies and singularities of tensor product surfaces of bidegree $(2,1)$} 

\author{Hal Schenck}
\thanks{Schenck supported by NSF 1068754, NSA H98230-11-1-0170}
\address{Department of Mathematics, University of Illinois, 
Urbana, IL 61801}
\email{schenck@math.uiuc.edu}

\author{Alexandra Seceleanu}
\address{Department of Mathematics, University of Nebraska,
    Lincoln, NE 68588}
\email{aseceleanu2@math.unl.edu}

\author{Javid Validashti}
\address{Department of Mathematics, University of Illinois, 
Urbana, IL 61801}
\email{jvalidas@illinois.edu}


\begin{abstract}
Let $U \subseteq H^0(\OPP(2,1))$ be a 
basepoint free four-dimensional vector space.
The sections corresponding to $U$ determine a regular map 
$\phi_U: \PP \longrightarrow \P^3$. We study the associated
bigraded ideal $I_U \subseteq k[s,t;u,v]$ from the standpoint
of commutative algebra, proving that there are exactly six
numerical types of possible bigraded minimal free resolution. These
resolutions play a key role in determining the implicit 
equation for $\phi_U(\PP)$, via work of Bus\'e-Jouanolou \cite{bj}, 
Bus\'e-Chardin \cite{bc}, Botbol \cite{bot} and 
Botbol-Dickenstein-Dohm \cite{bdd} 
on the approximation complex $\mathcal{Z}$. In four of the six cases 
$I_U$ has a linear first syzygy; remarkably from this we obtain all
differentials in the minimal free resolution. In particular this
allows us to explicitly describe the implicit equation and
singular locus of the image.
\end{abstract}

\maketitle

\section{Introduction}
A central problem in geometric modeling is to find simple 
(determinantal or close to it) equations for the image of 
a curve or surface defined by a regular or rational map. For 
surfaces the two most common situations are when  
$\PP \longrightarrow \P^3$ or $\P^2 \longrightarrow \P^3$. Surfaces
of the first type are called {\it tensor product surfaces} and
surfaces of the latter type are called {\it triangular surfaces}.
In this paper we study tensor product surfaces of bidegree
$(2,1)$ in  $\P^3$. The study of such surfaces goes back to the
last century--see, for example, works of Edge \cite{edge} and 
Salmon \cite{salmon}. 

Let $R = k[s,t,u,v]$ be a bigraded ring over an algebraically closed
field $k$, with $s,t$ of degree $(1,0)$ and $u,v$ of degree $(0,1)$.  
Let $R_{m,n}$ denote the graded piece in bidegree $(m,n)$.  A regular map 
$\PP \longrightarrow \P^3$ is defined by four  polynomials
\[
U= \spn \{ p_{0}, p_{1}, p_{2}, p_{3} \} \subseteq R_{m,n}
\]
with no common zeros on $\PP$. We will study the case $(m,n) = (2,1)$,
so
\[
U \subseteq H^0(\OPP(2,1)) = V = \spn \{ s^2u,stu,t^2u,s^2v,stv,t^2v \}.
\]
Let $I_U=\langle  p_{0}, p_{1}, p_{2}, p_{3}\rangle \subset R$,
$\phi_U$ be the associated map $\PP \longrightarrow \P^3$
and 
\[
X_U = \phi_U(\PP) \subseteq \P^3.
\]
We assume that $U$ is basepoint free, which means that 
\[
\sqrt{I_U} = \langle s,t\rangle \cap \langle u,v \rangle.
\]
We determine {\em all} possible numerical types of bigraded minimal free resolution for
$I_U$, as well as the embedded associated primes of $I_U$. 
Using approximation complexes, we relate the algebraic 
properties of $I_U$ to the geometry of $X_U$.
The next example illustrates our results.
\begin{exm}\label{ex1}
Suppose $U$ is basepoint free and $I_U$ has a unique
first syzygy of bidegree $(0,1)$. Then the primary 
decomposition of $I_U$ is given by Corollary~\ref{T5PD},
and the differentials in the bigraded minimal free resolution 
are given by Proposition~\ref{LS3}.  
For example, if $U = \spn \{s^2u, s^2v, t^2u, t^2v+stv \}$,
then by Corollary~\ref{T5PD} and Theorem~\ref{T5exact}, the
embedded primes of $I_U$ are $\langle s,t,u \rangle$ and 
$\langle s,t,v \rangle$, 
and by Proposition~\ref{LS3} the bigraded Betti numbers of $I_U$ are:
\[
0 \leftarrow I_U \leftarrow R(-2,-1)^4  \leftarrow 
R(-2,-2) \oplus R(-3,-2)^2 \oplus R(-4,-1)^2
\leftarrow  R(-4,-2)^2
\leftarrow 0
\]
Having the differentials in the free resolution allows us to 
use the method of approximation complexes to determine the
implicit equation: it follows from Theorem~\ref{ImX} that 
the image of $\phi_U$ is the hypersurface 
\[
X_U = \V(x_0x_1^2x_2-x_1^2x_2^2+2x_0x_1x_2x_3-x_0^2x_3^2).
\]
Theorem~\ref{SingX} shows that the reduced codimension one singular locus of $X_U$ is $\V(x_0,x_2) \cup \V(x_1,x_3) \cup \V(x_0,x_1)$.
\begin{figure}[ht]
\begin{center}
\includegraphics[width=.31\textwidth]{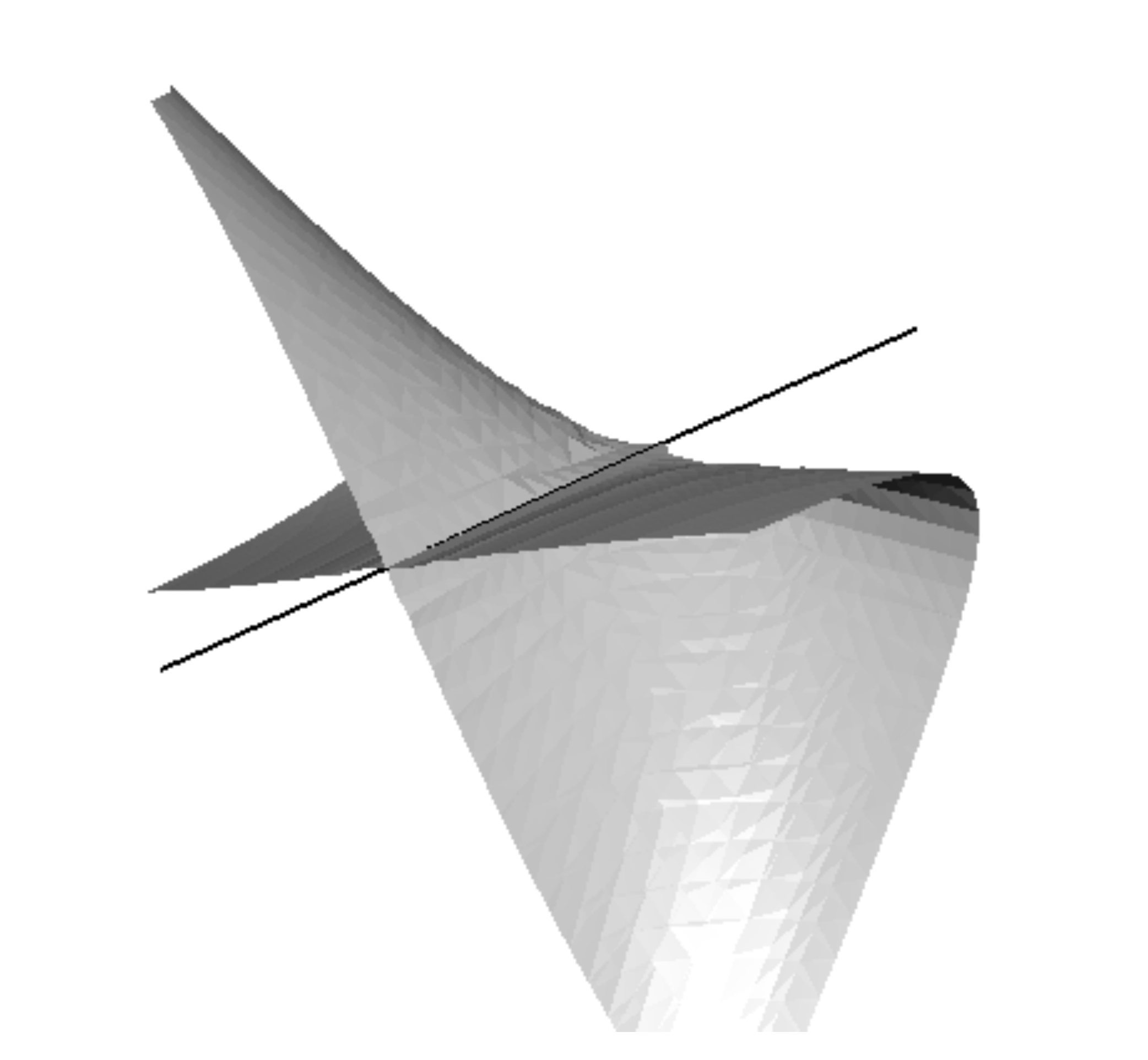}
\end{center}
\caption{\textsf{$X_U$ on the open set $U_{x_0}$}}
\label{fig:surfaceExample}
\end{figure}
The key feature of this example is that there is a linear syzygy of 
bidegree $(0,1)$:
\[
v \cdot (s^2u) -u \cdot(s^2v) =0.
\]
In Lemmas~\ref{LS1} and \ref{LS2-10} we show that with an appropriate
choice of generators for $I_U$, any bigraded linear first syzygy has the form above.
Existence of a bidegree $(0,1)$ syzygy implies that the 
pullbacks to $\PP$ of the two linear forms defining 
$\P(U)$ 
share a factor. Theorem~\ref{GPsyz} connects this to work of \cite{gl}.
\end{exm}
\subsection{Previous work on the $(2,1)$ case}
For surfaces in $\P^3$ of bidegree $(2,1)$, in addition to the 
classical work of Edge, Salmon and  others, more recently 
Degan \cite{d} studied such surfaces with basepoints and  
Zube \cite{z1}, \cite{z2} describes the possibilities for 
the singular locus. In \cite{egl}, Elkadi-Galligo-L\^{e} 
give a geometric description of the image 
and singular locus for a generic $U$ and  in \cite{gl}, 
Galligo-L\^{e} follow up with an analysis for the
nongeneric case. A central part of their analysis is 
the geometry of a certain dual scroll which we connect
to syzygies in \S 8.

Cox, Dickenstein and Schenck study the bigraded commutative algebra 
of a three dimensional basepoint free subspace
$W \subseteq R_{2,1}$ in \cite{cds}, showing that there are two 
numerical types of 
possible bigraded minimal free resolution of $I_W$, 
determined by how $\P(W) \subseteq \P(R_{2,1}) = \P^5$ 
meets the image $\Sigma_{2,1}$ of the Segre map 
$\P^2 \times  \P^1 \stackrel{\sigma_{2,1}}{\longrightarrow} \P^5.$
If $W$ is basepoint  free, then there are two possibilities: 
either $\P(W) \cap \Sigma_{2,1}$ is a finite set of points, or a 
smooth conic. The current paper extends the work of \cite{cds} to
the more complicated setting of a four dimensional space of sections.
A key difference is that for a basepoint free subspace $W$ of dimension
three, there can never be a linear syzygy on $I_W$. As illustrated
in the example above, this is not true for the four dimensional 
case. It turns out that the existence of a linear syzygy provides a
very powerful tool for analyzing both the bigraded commutative
algebra of $I_U$, as well as for determining the implicit 
equation and singular locus of 
$X_U$. In studying the bigraded commutative algebra of $I_U$, we employ
a wide range of tools
\begin{itemize}
\item Approximation complexes  \cite{bot}, \cite{bdd}, \cite{bj}, \cite{bc}, \cite{c}. 
\item Bigraded generic initial ideals \cite{ACD}.
\item Geometry of the Segre-Veronese variety \cite{h}.
\item Fitting ideals and Mapping cones \cite{ebig}.
\item Connection between associated primes and Ext modules \cite{ehv}.
\item Buchsbaum-Eisenbud exactness criterion \cite{be}.
\end{itemize}
\subsection{Approximation complexes}
The key tool in connecting the syzygies of $I_U$ to the implicit equation
for $X_U$ is an approximation complex, introduced by Herzog-Simis-Vasconcelos in
\cite{hsv1},\cite{hsv2}. We give more details of the 
construction in \S 7. The basic idea is as follows: 
let $R_I = R \oplus I_U \oplus I_U^2 \oplus \cdots$. Then the graph 
$\Gamma$ of the map $\phi_U$ is equal to $BiProj(R_I)$ and  the embedding
of $\Gamma$ in $(\PP) \times \P(U)$ corresponds to the ring map 
$S = R[x_0,\ldots,x_3] \stackrel{s}{\rightarrow} R_I$
given by $x_i \mapsto p_i$. Let $\beta$ denote the kernel of $s$, so $\beta_1$
consists of the syzygies of $I_U$ and  $S_I = Sym_R(I) = S/\beta_1$. 
Then 
\[
\Gamma \subseteq BiProj(S_I)  \subseteq BiProj(S).
\]
The works \cite{bdd}, \cite{bj}, \cite{bc}, \cite{bot} show 
that if $U$ is basepoint free, then the implicit equation for $X_U$ 
may be extracted from the differentials of a complex $\mathcal{Z}$ 
associated to the intermediate object $S_I$ and in particular the determinant of
the complex is a power of the implicit equation. 
In bidegree $(2,1)$, a result of Botbol \cite{bot} shows that 
the implicit equation may be obtained from a $4 \times 4$ minor 
of $d_1$; our work yields an explicit description of the relevant minor.
\subsection{Main results}
The following two tables describe our classification. 
Type refers to the graded Betti numbers of the bigraded minimal free 
resolution for $I_U$: we prove there are six numerical types possible. 
Proposition~\ref{PD2} shows that the only possible 
embedded primes of $I_U$ are $\mm = \langle s,t,u,v \rangle$
or $P_i = \langle l_i, s, t \rangle$, where $l_i$ is a linear
form of bidegree $(0,1)$. While Type 5a and 5b have the same bigraded Betti
numbers, Proposition~\ref{LS3} and Corollary~\ref{T5PD} show that
both the embedded primes and the 
differentials in the minimal resolution differ. 
We also connect our 
classification to the reduced, codimension one singular 
locus of $X_U$. In the table below $T$ denotes a twisted 
cubic curve, $C$ a smooth plane conic
and $L_i$ a line.
\begin{table}[ht]
\begin{center}
\begin{supertabular}{|c|c|c|c|c|}
\hline Type & Lin. Syz. & Emb. Pri. & Sing. Loc. & Example \\
\hline 1 & none             & $\mm$       & $T$                    &$\!s^2u\!+\!stv,t^2u,s^2v\!+\!stu,t^2v\!+\!stv\!$\\
\hline 2 & none             & $\mm, P_1$  & $C \cup L_1$            & $s^2u,t^2u,s^2v+stu,t^2v+stv$ \\
\hline 3 & 1 type $(1,0)$   & $\mm$       & $L_1$                   & $s^2u+stv,t^2u,s^2v,t^2v+stu$ \\
\hline 4 & 1 type $(1,0)$   & $\mm, P_1$  & $L_1$                  & $stv,t^2v,s^2v-t^2u,s^2u$ \\
\hline 5a & 1 type $(0,1)$   & $P_1, P_2$  & $L_1 \cup L_2 \cup L_3$  &  $s^2u, s^2v, t^2u, t^2v+stv$\\
\hline 5b & 1 type $(0,1)$   & $P_1$       & $L_1 \cup L_2$           & $s^2u, s^2v, t^2u, t^2v+stu$ \\
\hline 6 & 2 type $(0,1)$   & none        & $\emptyset$             & $s^2u,s^2v,t^2u,t^2v$ \\
\hline
\end{supertabular}
\end{center}
\caption{\textsf{}}
\label{T1}
\end{table}

\pagebreak
The next table gives the possible numerical types for the 
bigraded minimal free resolutions, where we write $(i,j)$ for the
rank one free module $R(i,j)$. We prove more: for Types 3, 4, 5 and 6, 
we determine all the differentials in the minimal free resolution.
One striking feature of Table 1 is that if $I_U$ has a linear first syzygy
(i.e. of bidegree $(0,1)$ or $(1,0)$), then the codimension one singular locus of
$X_U$ is either empty or a union of lines. We prove this in Theorem~\ref{SingX}.
\begin{table}[ht]
\begin{center}
\begin{supertabular}{|c|c|}
\hline Type & Bigraded Minimal Free Resolution of $I_U$ for $U$ basepoint free \\
\hline 1 &   $0 \leftarrow I_U \leftarrow (-2,-1)^4 \longleftarrow \begin{array}{c}
 (-2,-4)\\
 \oplus \\
(-3,-2)^4\\
 \oplus \\
(-4,-1)^2\\
\end{array} \longleftarrow 
\begin{array}{c}
 (-3,-4)^2\\
 \oplus \\
(-4,-2)^3\\
 \end{array} \longleftarrow 
(-4,-4)\leftarrow 0$ \\

\hline 2 &   $0 \leftarrow I_U \leftarrow (-2,-1)^4 \longleftarrow \begin{array}{c}
 (-2,-3)\\
 \oplus \\
(-3,-2)^4\\
 \oplus \\
(-4,-1)^2\\
\end{array} \longleftarrow 
\begin{array}{c}
 (-3,-3)^2\\
 \oplus \\
(-4,-2)^3\\
 \end{array} \longleftarrow 
(-4,-3) \leftarrow 0 $\\

\hline 3 & $0 \leftarrow I_U \leftarrow  (-2,-1)^4   \longleftarrow \begin{array}{c}
 (-2,-4)\\
 \oplus \\
(-3,-1)\\
 \oplus \\
(-3,-2)^2\\
\oplus \\
(-3,-3)\\
 \oplus \\
(-4,-2)\\
\oplus \\
(-5,-1)\\
\end{array} \longleftarrow 
\begin{array}{c}
 (-3,-4)^2\\
 \oplus \\
(-4,-3)^2\\
\oplus \\
(-5,-2)^2\\
 \end{array} \longleftarrow 
\begin{array}{c}
 (-4,-4)\\
 \oplus \\
(-5,-3)\\
 \end{array} \leftarrow 0 $\\

\hline 4 &  $0 \leftarrow I_U \leftarrow (-2,-1)^4 \longleftarrow \begin{array}{c}
 (-2,-3)\\
 \oplus \\
(-3,-1)\\
 \oplus \\
(-3,-2)^2\\
 \oplus \\
(-4,-2)\\
\oplus \\
(-5,-1)\\
\end{array} \longleftarrow 
\begin{array}{c}
 (-3,-3)\\
 \oplus \\
(-4,-3)\\
\oplus \\
(-5,-2)^2\\
 \end{array} \longleftarrow 
(-5,-3) \leftarrow 0$\\
 
\hline 5 &  \!\!\!\!\!\!\!\!\!\!\!\!\!\!\!\!\!\!\!\!\!\!\!\!\!\!\!\!\!\!\!\!\!\!\!\!\!\!\!\!\!$0 \leftarrow I_U \leftarrow (-2,-1)^4 \longleftarrow \begin{array}{c}
(-2,-2)\\
\oplus\\
(-3,-2)^2 \\
\oplus\\
(-4,-1)^2 \\
\end{array}
\longleftarrow  (-4,-2)^2
\leftarrow 0$ \\

\hline 6 & \!\!\!\!\!\!\!\!\!\!\!\!\!\!\!\!\!\!\!\!\!\!\!\!\!\!\!\!\!\!\!\!\!\!\!\!\!\!\!\!\!\!\!\!$0 \leftarrow I_U \leftarrow (-2,-1)^4 \longleftarrow \begin{array}{c}
 (-2,-2)^2\\
 \oplus \\
(-4,-1)^2\\
\end{array} \longleftarrow (-4,-2)\leftarrow 0$ \\
\hline
\end{supertabular}
\end{center}
\caption{\textsf{}}
\label{T2}
\end{table}
\vskip .1in

\section{Geometry and the Segre-Veronese variety}
Consider the composite maps
\begin{equation}\label{e1}
\xymatrix@C=8pt{
 {\PP} \ar[r] \ar[drr]^{\phi_U} &  \mathbb{P}(H^0(\OPone(2)))
\times  \mathbb{P}(H^0(\OPone(1))) \ar[r] &  \mathbb{P}(H^0(\OPP(2,1)))\ar[d]^{\pi}    \\
 & & \mathbb{P}(U)}
\end{equation}
The first horizontal map is $\nu_2 \times id$, where $\nu_2$ is the 
2-uple Veronese embedding and the second horizontal map is
the Segre map $\sigma_{2,1}: \mathbb{P}^2 \times \mathbb{P}^1 \rightarrow  
\mathbb{P}^5$. The image of $\sigma_{2,1}$ is a smooth 
irreducible nondegenerate cubic threefold $\Sigma_{2,1}$. 
Any $\P^2 \subseteq \Sigma_{2,1}$ is a fiber over a point
of the $\P^1$ factor and any $\P^1 \subseteq \Sigma_{2,1}$ is 
contained in the image of a fiber over $\P^2$ or
$\P^1$. For this see Chapter 2 of \cite{harris}, which also points
out that the Segre and Veronese maps have coordinate free descriptions
\[
\begin{array}{ccc}
\mathbb{P}(A) &\stackrel{\nu_d}{\longrightarrow}& \mathbb{P}(Sym^d A)\\
\mathbb{P}(A) \times \mathbb{P}(B) &\stackrel{\sigma}{\longrightarrow}& \mathbb{P}(A \otimes B)
\end{array}
\]
By dualizing we may interpret the image of $\nu_d$ as the 
variety of $d^{th}$ powers of linear forms on $A$ 
and the image of $\sigma$ as the variety of products of 
linear forms. The composition $\tau = \sigma_{2,1}\circ (\nu_2 \times id)$ 
is a Segre-Veronese map, with image consisting of polynomials which factor as 
$l_1(s,t)^2\cdot l_2(u,v)$. Note that $\Sigma_{2,1}$ is also 
the locus of polynomials in $R_{2,1}$ which factor as $q(s,t) \cdot l(u,v)$, with
$q \in R_{2,0}$ and $l \in R_{0,1}$. Since $q \in R_{2,0}$ factors as 
$l_1 \cdot l_2$, this means $\Sigma_{2,1}$ is the locus of polynomials 
in $R_{2,1}$ which factor completely as products of linear forms. 
As in the introduction, 
\[
U \subseteq H^0(\OPP(2,1)) = V = \spn \{ s^2u,stu,t^2u,s^2v,stv,t^2v \}.
\]
The ideal of $\Sigma_{2,1}$ is defined by the two by two minors of 
\[
\left[\begin{matrix} x_0 & x_1 & x_2 \\ x_3 & x_4 &x_5 
\end{matrix}\right].
\]
It will also be useful to understand the intersection of $\P(U)$
with the locus of polynomials in $V$ which factor as the product of a form 
$q =a_0su+a_1sv+a_2tu+a_3tv$ of bidegree  
$(1,1)$ and $l = b_0s+b_1t$ of bidegree $(1,0)$. This is the
image of the map 
\[
\P(H^0(\OPP(1,1))) \times \P(H^0(\OPP(1,0))) = \P^3 \times \P^1 \longrightarrow \P^5,
\]
$(a_0:a_1:a_2:a_3) \times (b_0:b_1) \mapsto 
(a_0b_0:a_0b_1+a_2b_0:a_2b_1:a_1b_0:a_1b_1+a_3b_0:a_3b_1)$, which
is a quartic hypersurface
\[
Q = \V({x}_{2}^{2} {x}_{3}^{2}-{x}_{1} {x}_{2} {x}_{3} {x}_{4}+{x}_{0} {x}_{2}
     {x}_{4}^{2}+{x}_{1}^{2} {x}_{3} {x}_{5}-2 {x}_{0} {x}_{2} {x}_{3} {x}_{5}-{x}_{0} {x}_{1} {x}_{4} {x}_{5}+{x}_{0}^{2} {x}_{5}^{2}).
\]

\noindent As Table 1 shows, the key to classifying the minimal 
free resolutions is understanding the {\em linear} syzygies.
In \S 3, we show that if $I_U$ has a first syzygy of bidegree $(0,1)$, then 
after a change of coordinates, 
$I_U = \langle pu, pv, p_2,p_3 \rangle$ and  if $I_U$ 
has a first syzygy of bidegree $(1,0)$, then 
$I_U = \langle ps, pt, p_2,p_3 \rangle$. 
\begin{prop}\label{LinSyzGeom}
If $U$ is basepoint  free, then the ideal $I_U$ 
\begin{enumerate}
\item has a unique linear syzygy of bidegree $(0,1)$ iff
$F \subseteq \P(U) \cap \Sigma_{2,1}$, where $F$ is a $\P^1$ fiber of
$\Sigma_{2,1}$.
\item has a pair of linear syzygies of bidegree $(0,1)$
iff $\P(U) \cap \Sigma_{2,1} = \Sigma_{1,1}$.
\item has a unique linear syzygy of bidegree $(1,0)$ 
iff $F \subseteq \P(U) \cap Q$, where $F$ is a $\P^1$ fiber of
$Q$.
\end{enumerate}
\end{prop}
\begin{proof}
The ideal $I_U$ has a unique linear syzygy of bidegree $(0,1)$ iff
$qu,qv \in I_U$, with $q \in R_{2,0}$ iff $q\cdot l(u,v) \in I_U$ for all
$l(u,v) \in R_{0,1}$ iff $\P(U) \cap \Sigma_{2,1}$ contains the 
$\P^1$ fiber over the point $q \in \P(R_{2,0})$. 

For the second item, the reasoning above implies that 
$\P(U) \cap \Sigma_{2,1}$ contains two $\P^1$ fibers, over 
points $q_1,q_2 \in \P(R_{2,0})$. But then $I_U$ also contains
the line in $\P(R_{2,0})$ connecting $q_1$ and $q_2$, as well
as the $\P^1$ lying over any point on the line, yielding a
$\P^1 \times \P^1$. 

For the third part, a linear syzygy of bidegree $(1,0)$ means
that $qs,qt \in I_U$, with $q \in R_{1,1}$ iff $q\cdot l(s,t) \in I_U$ for all
$l(s,t) \in R_{1,0}$ iff $\P(U) \cap Q$ contains the 
$\P^1$ fiber over the point $q \in \P(R_{1,1})$.  
\end{proof}
\noindent In Theorem~\ref{AllLins}, we show that Proposition~\ref{LinSyzGeom}
describes all possible linear syzygies.

\section{First syzygies of bidegree $(0,1)$}
\noindent Our main result in this section is a complete description 
of the minimal free resolution when $I_U$ has a first 
syzygy of bidegree $(0,1)$. As a consequence, if $I_U$ has a unique first
syzygy of bidegree $(0,1)$, then the minimal free resolution has 
numerical Type 5 and  if there are two linear first syzygies of 
bidegree $(0,1)$, the minimal free resolution has numerical Type 6. 
We begin with a simple observation
\begin{lem}\label{LS1}
If $I_U$ has a linear first syzygy of bidegree $(0,1)$, then 
\[
I_U = \langle pu, pv, p_2,p_3 \rangle,
\]
where $p$ is homogeneous of bidegree $(2,0)$.
\end{lem}
\begin{proof}
Rewrite the syzygy 
\[
\sum\limits_{i=0}^3 (a_iu+b_iv)p_i = 0 = u \cdot \sum\limits_{i=0}^3 a_ip_i + v\cdot \sum\limits_{i=0}^3 b_ip_i,
\]
and let $g_0 = \sum\limits_{i=0}^3 a_ip_i$, $g_1 =\sum\limits_{i=0}^3 b_ip_i$.
The relation above implies that $(g_0,g_1)$ is a syzygy on $(u,v)$.
Since the syzygy module of $(u,v)$ is generated by the Koszul syzygy,
this means
\[
\left[ \!
\begin{array}{c}
g_0\\
g_1
\end{array}\! \right] = p \cdot \left[ \!
\begin{array}{c}
-v\\
u
\end{array}\! \right]
\]
\end{proof}
A similar argument applies if $I_U$ has a first syzygy of
degree $(1,0)$. Lemma~\ref{LS1} has surprisingly strong consequences:
\begin{prop}\label{LS3}
If $U$ is basepoint free and $I_U$ has a unique linear first syzygy of bidegree $(0,1)$, then 
there is a complex of free $R$ modules 
\[
\mathcal{F}_1 \mbox{ : }0 \longrightarrow F_3 \stackrel{\phi_3}{\longrightarrow}  F_2 \stackrel{\phi_2}{\longrightarrow} F_1\stackrel{\phi_1}{\longrightarrow} I_U  \longrightarrow 0,
\]
where $\phi_1 = \left[ \!\begin{array}{cccc}
p_0 & p_1 & p_2 & p_3 \end{array}\! \right]$, with ranks and shifts matching Type 5 in Table 2. Explicit formulas
appear in the proof below. The differentials $\phi_i$ depend on whether
$p=L_1(s,t)L_2(s,t)$ of Lemma~\ref{LS1} has $L_1=L_2$.
\end{prop}
\begin{proof}
Since $I_U$ has a syzygy of bidegree $(0,1)$, by Lemma~\ref{LS1}, $I_U = \langle pu, pv, p_2,p_3 \rangle$.\newline
Case 1: Suppose $p = l(s,t)^2$, then 
after a change of coordinates, $p = s^2$, so 
$p_0 = s^2u$ and $p_1= s^2v$. Eliminating terms from $p_2$ and $p_3$,
we may assume 
\[
\begin{array}{ccc}
p_2 &=& stl_1(u,v)+t^2l_2(u,v)=t(sl_1(u,v)+tl_2(u,v))\\
p_3 &=& stl_3(u,v)+t^2l_4(u,v)=t(sl_3(u,v)+tl_4(u,v)).
\end{array}
\]
Let $l_i(u,v) = a_iu+b_iv$ and  define
\[
A(u,v) = \left[ \!
\begin{array}{cc}
l_1 & l_2\\
l_3 & l_4
\end{array}\! \right]. 
\]
Note that det $A(u,v) = q(u,v) \ne 0$. The rows cannot be dependent, 
since $U$ spans a four dimensional subspace. If the columns are 
dependent, then $\{p_2,p_3 \} = \{tl_1(s+kt), tl_2(s+kt)\}$,
yielding another syzygy of bidegree $(0,1)$, contradicting our hypothesis.
In the proof of Corollary~\ref{NoMix}, we show the hypothesis that
$U$ is basepoint free implies that $A(u,v)$ is a $1$-generic matrix,
which means that $A(u,v)$ cannot be made to have a zero entry using 
row and column operations. We obtain a first syzygy
of bidegree $(2,0)$ as follows:
\[
\begin{array}{ccl}
s^2p_2 &=&s^3tl_1 + s^2t^2l_2\\
      &=& (a_1st+a_2t^2)s^2u + (b_1st+b_2t^2)s^2v\\
      &=& (a_1st+a_2t^2)p_0 + (b_1st+b_2t^2)p_1\\
\end{array}
\]
A similar relation holds for $stp_3$, yielding two first syzygies
of bidegree $(2,0)$. We next consider first syzygies of bidegree $(1,1)$. There is an obvious syzygy on $p_2,p_3$ given by
\[
\begin{array}{ccccc}
(sl_1(u,v)+tl_2(u,v))p_3 & = & (sl_3(u,v)+tl_4(u,v))p_2 
\end{array}
\]
Since  $\DET A(s,t) = q(u,v) \ne 0$, from 
\[
\begin{array}{ccc}
t^2q &=& l_3p_2-l_1p_3\\
stq &=& l_4p_2-l_2p_3
\end{array}
\]
and the fact that $q(u,v) = L_1(u,v)L_2(u,v)$ with
$L_i(u,v) = \alpha_iu+\beta_iv$, so we obtain a pair of relations
of bidegree $(1,1)$:
\[
\begin{array}{ccccc}
sl_4p_2-sl_2p_3 &= & s^2tL_1L_2 &=& (\alpha_1tL_2)s^2u+(\beta_1tL_2)s^2v. 
\end{array}
\]
Case 2: $p = l(s,t) \cdot l'(s,t)$ with $l,l'$ independent
linear forms. Then after a change of coordinates, $p = st$, so 
$p_0 = stu$ and $p_1= stv$. Eliminating terms from $p_2$ and $p_3$,
we may assume 
\[
\begin{array}{ccc}
p_2 &=& s^2l_1(u,v)+t^2l_2(u,v)\\
p_3 &=& s^2l_3(u,v)+t^2l_4(u,v).
\end{array}
\]
Let $l_i(u,v) = a_iu+b_iv$. We obtain a first syzygy
of bidegree $(2,0)$ as follows:
\[
\begin{array}{ccc}
stp_2 &=&  \!\!\!\!\!\!\!\!\!\!\!\!\!\!\!\!\!\!\!\!\!\!\!\!\!\!\!\!\!\!\!\!\!\!\!\!\!\!\!\!\!\!\!\!\!\!\!\!\!\!\!\!\!\!\!\!s^3tl_1 + st^3l_2\\
      &=& \!\!\!\!\!\!\!\!\!\!\!\!\!\!\!\!\!\!\!\!\!\!\!\!\!\!\!\! \!\!\!\!\!\!\!\!\!\!\!\!\!\!s^2(stl_1) + t^2(stl_2)\\
      &=& (a_1s^2+a_2t^2)stu + (b_1s^2+b_2t^2)stv
\end{array}
\]
A similar relation holds for $stp_3$, yielding two first syzygies
of bidegree $(2,0)$. We next consider first syzygies of bidegree $(1,1)$.
Since $ q(u,v) \ne 0$, from
\[
\begin{array}{ccc}
t^2q &=& l_3p_2-l_1p_3\\
s^2q &=& l_4p_2-l_2p_3
\end{array}
\]
and the fact that $q(u,v) = L_1(u,v)L_2(u,v)$ with
$L_i(u,v) = \alpha_iu+\beta_iv$, we have relations
\[
\begin{array}{ccccc}
sl_3p_2-sl_1p_3 &= & st^2L_1L_2 &=& (\alpha_1tL_2)stu+(\beta_1tL_2)stv \\
tl_4p_2-tl_2p_3 &= & ts^2L_1L_2 &=& (\alpha_1sL_2)stu+(\beta_1sL_2)stv, 
\end{array}
\]
which yield a pair of first syzygies of bidegree $(1,1)$. 
Putting everything together, we now have candidates for 
the differential $\phi_2$ in both cases. Computations
exactly like those above yield similar candidates for
$\phi_3$ in the two cases. In Case 1, we have 
\[
\phi_2 = \left[ \!
\begin{array}{ccccc}
v & \alpha_1tL_2 &   0& a_1st+a_2t^2  & a_3st+a_4t^2  \\
-u  &\beta_1tL_2 & 0 & b_1st+b_2t^2 &  b_3st+b_4t^2  \\
0     & -sl_4 & sl_3+tl_4 & -s^2   &0 \\
0  & sl_2 & -sl_1-tl_2 &0  &  -s^2 
\end{array}\! \right], \mbox{ }
\phi_3 = \left[ \!
\begin{array}{cc}
\gamma & \delta \\
s & t\\
0 & s\\
-l_4 & -l_3\\
l_2 & l_1
\end{array}\! \right],
\]
where 
\[
\begin{array}{lll}
\delta &= & -\alpha_1\beta_2 t^2 + (a_1st+a_2t^2)b_3 - (a_3st+a_4t^2)b_1\\
\gamma & = &-\alpha_1\beta_2st + (a_1st+a_2t^2)b_4 -(a_3st+a_4t^2)b_2 
\end{array}
\]

For $I_U$ as in Case 2, let 
\[
\phi_2 = \left[ \!
\begin{array}{ccccc}
v & \alpha_1tL_2 &   \alpha_1sL_2 & a_1s^2+a_2t^2  & a_3s^2+a_4t^2  \\
-u  &\beta_1tL_2 & \beta_1sL_2 & b_1s^2+b_2t^2 &  b_3s^2+b_4t^2  \\
0     & -sl_3 & -tl_4 & -st   &0 \\
0  & sl_1 & tl_2 &0  &  -st 
\end{array}\! \right], \mbox{ }
\phi_3 = \left[ \!
\begin{array}{cc}
\gamma & \delta \\
0 & t\\
s & 0\\
-l_4 & -l_3\\
l_2 & l_1
\end{array}\! \right],
\]
where 
\[
\begin{array}{ccc}
\gamma & = &  (-\alpha_1\beta_2  + a_1b_4 - a_3b_2)s^2 + (a_2b_4 - a_4b_2)t^2\\
\delta & = &  (a_1b_3 -a_3b_1)s^2 + (-\alpha_1\beta_2+a_2b_3 -a_4b_1)t^2
\end{array}
\]
We have already shown that $\im(\phi_2) \subseteq \ker(\phi_1)$, 
and an easy check shows that $\im(\phi_3) \subseteq \ker(\phi_2)$, yielding 
a complex of free modules of numerical Type 5.
\end{proof}
\noindent To prove that the complex above is actually exact, we use the
following result of Buchsbaum and Eisenbud \cite{be}: 
a complex of free modules 
\[\mathcal{F}: \cdots \stackrel{\phi_{i+1}}{\longrightarrow}  F_i \stackrel{\phi_i}{\longrightarrow} F_{i-1}\stackrel{\phi_{i-1}}{\longrightarrow}\cdots F_1\stackrel{\phi_1}{\longrightarrow}F_0,
\]
is exact iff
\begin{enumerate}
\item $\rank(\phi_{i+1}) + \rank(\phi_i) = \rank(F_i)$.
\item $\depth(I_{\rank(\phi_i)}(\phi_i)) \ge i$.
\end{enumerate}
\begin{thm}\label{T5exact}
The complexes appearing in Proposition~\ref{LS3} are exact.
\end{thm}
\begin{proof} 
Put $F_0=R$. An easy check shows that $\rank(\phi_{i+1}) + \rank(\phi_i) = \rank(F_i)$, so what remains is to show that $\depth(I_2(\phi_3)) \ge 3$ and $\depth(I_3(\phi_2)) \ge 2$. The fact that $s \nmid \gamma$ will be useful: to see
this, note that both Case 1 and Case 2, $s \mid \gamma$ iff $a_2b_4-b_2a_4 =0$,
which implies $l_2$ and $l_4$ differ only by a scalar, contradicting the
assumption that $U$ is basepoint free. \newline
Case 1: We have $us^4, vs^4 \in I_3(\phi_2)$. Consider the minor 
$$\lambda=t^2L_2(sl_1+tl_2)\left( (\alpha_1b_1-\beta_1a_1)s + (\alpha_1b_2-\beta_1a_2)t \right)$$ obtained from the submatrix
$$
\left[
\begin{array}{ccc}
 \alpha_1tL_2 &  0 & a_1st+a_2t^2    \\
\beta_1tL_2 &0 & b_1st+b_2t^2   \\
sl_2 & -sl_1-tl_2 &0  
\end{array}\! \right]
$$
Note that $s$ does not divide $\lambda$, for if $s | \lambda$ then  either $l_2=0$ or $\alpha_1b_2-\beta_1a_2=0$. But none of the $l_i$ can be zero because of the basepoint free assumption (see the proof of  Corollary \ref{NoMix}) and  if $\alpha_1b_2-\beta_1a_2=0$, then $L_1$ and $l_2$ are the same up to a scalar multiple. Hence, since $l_1l_4-l_2l_3=L_1L_2$, we obtain $l_1$ is equal to a scalar multiple of $l_2$ or  $l_3$, which again violates the basepoint free assumption. 
To conclude, note that  $u$ and $v$ can not divide $\lambda$ at the same time, therefore, $\lambda$ and one of the $ us^4$ and $vs^4$ form a regular sequence in $I_3(\phi_2)$, showing that depth of $I_3(\phi_2)$ is at least $2$.\\
To show that $\depth(I_2(\phi_3)) \ge 3$, note that 
\[
I_2(\phi_3) = \langle sl_2,sl_4,tl_4-sl_3,tl_2-sl_1,t\gamma-s\delta, 
s\gamma, s^2, q(u,v) \rangle
\]
Since $l_2, l_4$ are independent, 
$\langle sl_2,sl_4 \rangle = \langle su,sv \rangle$ and  using these we
can reduce $tl_4-sl_3,tl_2-sl_1$ to $tu, tv$. Since 
$s \nmid \gamma$, modulo $s^2$, $s\gamma$ reduces to $st^2$. 
Similarly, $t\gamma-s\delta$ reduces to $t^3$, so that in fact
\[
I_2(\phi_3)  = \langle su,sv,tu,tv,s^2,st^2,t^3,q(u,v) \rangle, 
\]
and $\{s^2, t^3, q(u,v)\}$ is a regular sequence of length three.\\
Case 2:  We have $us^2t^2, vs^2t^2 \in I_3(\phi_2)$. Consider the minor 
\[
\lambda=L_2(s^2l_1-t^2l_2)\left( (\alpha_1b_1-\beta_1a_1)s^2 + (\alpha_1b_2-\beta_1a_2)t^2 \right)
\]
arising from the submatrix
$$
\left[
\begin{array}{ccc}
 \alpha_1tL_2 &   \alpha_1sL_2 & a_1s^2+a_2t^2    \\
\beta_1tL_2 & \beta_1sL_2 & b_1s^2+b_2t^2   \\
sl_1 & tl_2 &0  
\end{array}\! \right]
$$
Note that $s$ and $t$ do not divide $\lambda$, for if $s | \lambda$ then  either $l_2=0$ or $\alpha_1b_2-\beta_1a_2=0$. But none of the $l_i$ can be zero because of the basepoint free assumption (see the proof of  Corollary \ref{NoMix}) and  if $\alpha_1b_2-\beta_1a_2=0$, then $L_1$ and $l_2$ are the same up to a scalar multiple. Hence, since $l_1l_4-l_2l_3=L_1L_2$, we obtain $l_1$ is equal to a scalar multiple of $l_2$ or  $l_3$, contradicting basepoint freeness. Furthermore, 
$u$ and $v$ cannot divide $\lambda$ at the same time, so $\lambda$ and one of the $ us^2t^2$ and $vs^2t^2$ form a regular sequence in $I_3(\phi_2)$.
To show that $\depth(I_2(\phi_3)) \ge 3$, note that 
\[
I_2(\phi_3) = \langle su,sv,tu,tv,st, t\gamma, s\delta, q(u,v) \rangle,
\]
where we have replaced $sl_i, tl_j$ as in Case 1. If $t \mid \delta$, 
then $a_1b_3-b_1a_3 = 0$, which would mean $l_1 = k l_3$ and contradict 
that $U$ is basepoint free. Since $s \nmid \gamma$ and $t \nmid \delta$, 
$\{t\gamma, s\delta, q(u,v) \}$ is regular unless $\delta$, $\gamma$ share a
common factor $\eta = (as+bt)$. Multiplying out and comparing coefficients
shows that this forces $\gamma$ and $\delta$ to agree up to scalar. 
Combining this with the fact that $t \nmid \delta$, $s \nmid \delta$, 
we find that $\delta = as^2+bst+ct^2$ with $a \ne 0 \ne c$. Reducing
$s \delta$ and $t\delta$ by $st$ then implies that $t^3,s^3 \in I_2(\phi_3)$.
\end{proof}

\begin{cor}\label{NoMix}
If $U$ is basepoint  free, then $I_U$ cannot have first syzygies of both bidegree $(0,1)$ and bidegree $(1,0)$.
\end{cor}
\begin{proof}
Suppose there is a first syzygy of bidegree $(0,1)$ and  proceed
as in the proof of Proposition~\ref{LS3}. In the setting of 
Case 1, 
\[
I_U =\langle  stu, stv, s^2l_1(u,v)+t^2l_2(u,v), s^2l_3(u,v)+t^2l_4(u,v)\rangle .
\]
If there is also a linear syzygy of bidegree $(1,0)$, expanding out
$\sum (a_is+b_it)p_i$ shows that the coefficient of $s^3$ is $a_3l_1+a_4l_3$,
and the coefficient of $t^3$ is $b_3l_2+b_4l_4$. Since 
$\sum (a_is+b_it)p_i=0$, both coefficients must vanish.
In the proof of Proposition~\ref{LS3} we showed $\DET A(u,v) = q(u,v) \ne 0$.
In fact, more is true: if any of the $l_i$ is zero, then $U$ is not basepoint 
free. For example, if $l_1 =0$, then $\langle t,l_3\rangle$ is a minimal 
associated prime of $I_U$. Since $a_3l_1+a_4l_3 =0$ iff $a_3=a_4=0$ or 
$l_1$ is a scalar multiple of $l_3$ and  the latter situation implies
that $U$ is not basepoint  free, we must have $a_3=a_4=0$. Reasoning 
similarly for $b_3l_2+b_4l_4$ shows that $a_3=a_4=b_3=b_4=0$. This implies
the linear syzygy of bidegree $(1,0)$ can only involve $stu,stv$, which 
is impossible. This proves the result in Case 1 and  
similar reasoning works for Case 2.
\end{proof}

\begin{cor}\label{T5PD}
If $I_U$ has a unique linear first syzygy of bidegree $(0,1)$, then $I_U$
has either one or two embedded prime ideals of the form $\langle s,t, L_i(u,v) \rangle$. If $q(u,v) = \DET~A(u,v) = L_1(u,v)L_2(u,v)$ for $A(u,v)$ as in Theorem~\ref{T5exact},
then:
\begin{enumerate}
\item If $L_1=L_2$, then the only embedded prime of $I_U$ is 
$\langle s,t,L_1\rangle$.
\item If $L_1 \ne L_2$, then $I_U$ has two embedded primes
$\langle s,t,L_1\rangle$ and $\langle s,t,L_2\rangle$.
\end{enumerate}
\end{cor}
\begin{proof}
In \cite{ehv}, Eisenbud, Huneke and  Vasconcelos show that 
a prime $P$ of codimension $c$ is associated to $R/I$ iff 
it is associated to $Ext^c(R/I,R)$. If $I_U$ has a unique 
linear syzygy, then the free resolution is given by Proposition~\ref{LS3},
and $Ext^3(R/I_U,R) = \coker(\phi_3^t)$. By Proposition 20.6 of
\cite{ebig}, if $\phi$ is a presentation matrix for a module $M$,
then the radicals of $ann(M)$ and $I_{\rank(\phi)}(\phi)$ are equal.
Thus, if $I_U$ has a Type 5 resolution, the codimension three
associated primes are the codimension three associated primes 
of $I_2(\phi_3)$. The proof of Theorem~\ref{T5exact} shows that
in Case 1,
\[
\begin{array}{cccc}
I_2(\phi_3) & = & \!\!\!\!\!\!\!\!\!\!\!\!\!\!\! \langle su,sv,tu,tv,s^2,st^2,t^3,q(u,v) \rangle &\\
           & =  & \!\!\!\!\!\!\!\!\!\!\!\!\!\!\!\!\!\!\!\!\!\!\!\!\!\!\!\!\! \langle s^2,st^2,t^3,u,v\rangle \cap \langle s,t,L_1^2 \rangle &\mbox{ if } L_1=L_2 \\
           & =  & \langle s^2,st^2,t^3,u,v\rangle \cap \langle s,t,L_1 \rangle 
\cap \langle s,t,L_2 \rangle &\mbox{ if } L_1\ne L_2 .\\
\end{array}
\]
The embedded prime associated to $\langle s,t,u,v \rangle$ is not
an issue, since we are only interested in the codimension three associated
primes. The proof for Case 2 works in the same way. 
\end{proof}
Next, we tackle the case where the syzygy of bidegree $(0,1)$ is 
not unique.
\begin{prop}\label{LS2}
If $U$ is basepoint free, then the following are equivalent
\begin{enumerate}
\item The ideal $I_U$ has two linear first syzygies of bidegree $(0,1)$.
\item The primary decomposition of $I_U$ is
\[
I_U = \langle u,v \rangle \cap  \langle q_1,q_2 \rangle,
\]
where $\sqrt{\langle q_1,q_2 \rangle} = \langle s,t \rangle$
and $q_i$ are of bidegree $(2,0)$. 
\item The minimal free resolution of $I_U$ is of numerical Type 6. 
\item $X_U \simeq \Sigma_{1,1}$.
\end{enumerate}
\end{prop}
\begin{proof}
By Lemma~\ref{LS1}, since $I_U$ has a linear syzygy of bidegree $(0,1)$,
\[
I_U = \langle q_1u, q_1v, p_2,p_3 \rangle.
\]
Proceed as in the proof of Proposition~\ref{LS3}. In 
Case 1, the assumption that $q_1=st$ means that $p_2,p_3$
can be reduced to have no terms involving $stu$ and $stv$, hence there
cannot be a syzygy of bidegree $(0,1)$ involving $p_i$ and
$q_1u,q_1v$. Therefore the second first syzygy of bidegree $(0,1)$ 
involves only $p_2$ and $p_3$ and  the reasoning in the 
proof of Lemma~\ref{LS1} implies that $\langle p_2,p_3 \rangle = \langle  q_2u, q_2v \rangle$. Thus, we have the primary decomposition
\[
I_U = \langle q_1,q_2 \rangle \cap \langle u, v \rangle,
\]
with $q_1,q_2$ of bidegree $(2,0)$. Since $U$ is basepoint  free,
$\sqrt{q_1,q_2} = \langle s,t \rangle$,
so $q_1$ and $q_2$ are a regular sequence in $k[s,t]$.
Similar reasoning applies in the situation of Case 2.
That the minimal free resolution is of numerical Type 6 
follows from the primary decomposition above,
which determines the differentials in 
the minimal free resolution:

\begin{small}
\[
I_U \longleftarrow (-2,-1)^4 \xleftarrow{\left[ \!
\begin{array}{cccc}
v & 0  &q_2  &0 \\
-u &0  &0 &  q_2 \\
0 & v  &-q_1  &0   \\
0 & -u &0  & -q_1
\end{array}\! \right]} 
\begin{array}{c}
(-2,-2)^2\\
\oplus \\
(-4,-1)^2
\end{array}\!
\xleftarrow{\left[ \!\begin{array}{c}
q_2 \\
-q_1\\
-v \\
u
\end{array}\! \right]}
 (-4,-2).
\]
\end{small}
The last assertion follows since 
\[
\DET \left[ \!
\begin{array}{cc}
uq_1 & uq_2\\
vq_1 & vq_2
\end{array}\! \right] = 0. 
\]
Hence, the image of $\phi_U$ is contained in $\V(xy-zw) = \Sigma_{1,1}$.
After a change of coordinates, $q_1=s^2+ast$ and $q_2 = t^2+bst$, with
$0 \ne a \ne b \ne 0$. Therefore on the open set 
$U_{s,u} \subseteq \PP$ the map is defined by
\[
(at+1,(at+1)v, t^2+bt, (t^2+bt)u), 
\]
so the image is a surface. Finally, if $X_U = \Sigma_{1,1}$, then 
with a suitable choice of basis for $U$, $p_0p_3-p_1p_2 = 0$, hence
$p_0 | p_1p_2$ and $p_3 | p_1p_2$. Since $U$ is four dimensional, 
this means we must have $p_0 = \alpha \beta$, $p_1 = \alpha \gamma$,
$p_2 = \beta \delta$. Without loss of generality, suppose $\beta$ is quadratic, so 
there is a linear first syzygy $\delta p_0-\alpha p_2 = 0$. 
Arguing similarly for $p_3$, we find that
there are two independent linear first syzygies. Lemma~\ref{LS5-10} of the
next section shows that if $U$ is basepoint free, then there can 
be at most one first syzygy of bidegree $(1,0)$, so by 
Corollary~\ref{NoMix}, $I_U$ must have two first syzygies of
bidegree $(0,1)$.
\end{proof}

\section{First syzygies of bidegree $(1,0)$}
\noindent Recall that there is an analogue of Lemma \ref{LS1} for syzygies of bidegree $(1,0)$:
\begin{lem}\label{LS2-10}
If $I_U$ has a linear syzygy of bidegree $(1,0)$, then 
\[
I_U = \langle ps, pt, p_2,p_3 \rangle,
\]
where $p$ is homogeneous of bidegree $(1,1)$.
\end{lem}
 Lemma~\ref{LS2-10} has  strong consequences as well: we will prove that
\begin{prop}\label{LS3-10}
If $U$ is basepoint free and $I_U= \langle ps, pt, p_2,p_3 \rangle$, then 
\begin{enumerate}
\item $I_U$ has numerical Type 4 if and only if $p$ is decomposable.
\item $I_U$ has numerical Type 3 if and only if $p$ is indecomposable.
\end{enumerate}
\end{prop}
We begin with some preliminary lemmas:
\begin{lem}\label{LS4-10}
If $I_U$ has a first syzygy of bidegree $(1,0)$, then 
$I_U$ has two minimal syzygies of bidegree $(1,1)$ and  
if $p$ in Lemma~\ref{LS2-10} factors, then $I_U$ also has a 
minimal first syzygy of bidegree $(0,2)$.
\end{lem}
\begin{proof}
First assume $p$ is an irreducible bidegree $(1,1)$ form, then $p = a_0su+a_1sv+a_2tu+a_3tv$, with $a_0a_3-a_1a_2 \ne 0$. We may assume $a_0\ne 0$ and 
scale so it is one. Then 
\[
\begin{array}{ccc}
sp & = & s^2u+a_1s^2v+ a_2stu+a_3stv \\
tp & = & stu + a_1stv+a_2t^2u+a_3t^2v \\
p_2 & = & b_0t^2u+b_1s^2v+b_2stv+b_3t^2v\\
p_3 & = & c_0t^2u+c_1s^2v+c_2stv+c_3t^2v\\
\end{array}
\]
Here we have used $tp$ and $sp$ to remove all the terms involving
$s^2u$ and $stu$ from $p_2$ and $p_3$. A simple but tedious calculation
then shows that 
\[
\begin{array}{ccc}
p \cdot p_2 & = & sp(b_1sv+b_2tv)+ tp(b_0tu+b_3tv)\\
p \cdot p_3 & = & sp(c_1sv+c_2tv)+ tp(c_0tu+c_3tv)
\end{array}
\]

Now suppose that $p = L_1(s,t) \cdot L_2(u,v)$ with $L_1, L_2$ 
linear forms, then after a change of coordinates, $p = su$ and  a (possibly new) set of minimal generators for $I_U$ is  $\langle s^2u, stu, p_2,p_3 \rangle$.
Eliminating terms from $p_2$ and $p_3$, we may assume 
\[
\begin{array}{ccc}
p_2 &=& as^2v+bstv+t^2l_1\\
p_3 &=& cs^2v+dstv+t^2l_2,
\end{array}
\]
where $l_i =l_i(u,v) \in R_{(0,1)}$. There are two first syzygies of bidegree $(1,1)$:
\[
\begin{array}{ccc}
sup_2 &= & su(as^2v+bstv+t^2l_1)=as^3uv+bs^2tuv+st^2ul_1 \\
&=& (as+bt)v\cdot s^2u+tl_1\cdot stu =(as+bt)v\cdot p_0+tl_1\cdot p_1\\
\end{array}
\]

\[
\begin{array}{ccc}
sup_3 &= & su(cs^2v+dstv+t^2l_2)=cs^3uv+ds^2tuv+st^2ul_2 \\
&=& (cs+dt)v\cdot s^2u+tl_2\cdot stu =(cs+dt)v\cdot p_0+tl_2\cdot p_1\\
\end{array}
\]
A syzygy of bidegree $(0,2)$ is obtained via:
\[
\begin{array}{ccc}
u(l_2p_3-l_1p_2) &= & u(as^2vl_2+bstvl_2-cs^2vl_1-dstvl_1) \\
&=& (al_2-cl_1)v\cdot s^2u+(bl_2-dl_1)v\cdot stu \\
&=& (al_2-cl_1)v\cdot p_0+(bl_2-dl_1)v\cdot p_1 \\
\end{array}
\]
\end{proof}

\begin{lem}\label{LS5-10}
If $U$ is basepoint free, then there can be at most one 
linear syzygy of bidegree $(1,0)$.
\end{lem}
\begin{proof}
Suppose $I_U$ has a linear syzygy of bidegree $(1,0)$, so that 
$sp, tp \in I_U$, with $p=su+a_1sv+a_2tu+a_3tv$. Note this takes
care of both possible cases of Lemma~\ref{LS4-10}: in Case 1,
$a_3-a_1a_2 =0$ (e.g. for $p = su$, $a_1=a_2=a_3 =0$) and  in 
Case 2, $a_3-a_1a_2 \ne 0$. Now suppose another syzygy of 
bidegree $(1,0)$ exists: $S = \sum (d_is+e_it)p_i = 0$. Expanding shows that
\[
S = d_0s^3u + (e_0+d_1)s^2tu + \cdots
\]
So after reducing $S$ by the Koszul syzygy on $\langle sp,tp \rangle$, 
$d_0=d_1=e_0 = 0$.
In $p_2$ and $p_3$, one of $b_1$ or $c_1$ must be non-zero. If 
not, then all of $tp, p_2,p_3$ are divisible by $t$ and  since 
\[
sp\mid_{t=0} = s^2(u+a_1v),
\]
this would mean $\langle t, u+a_1v \rangle$ is an associated prime of
$I_U$, contradicting basepoint  freeness. WLOG $b_1 \ne 0$, scale it
to one and use it to remove the term $c_1s^2v$ from $p_3$. This means
the coefficient of $s^3v$ in $S$ is $d_2b_1 = d_2$, so $d_2$ vanishes.
At this point we have established that $S$ does not involve $sp$, 
and involves only $t$ on $tp, p_2$. Now change generators so that 
$p_2' = e_1pt+e_2p_2$. This modification does not affect $tp,sp$ 
and $I_U$, but now $S$ involves only $p'_2$ and $p_3$:
\[
(t)p'_2 + (d_3s+e_3t)p_3 = 0.
\]
As in the proof of Lemma~\ref{LS1}, letting $p_2''=p_2'+e_3p_3$ and
$p_3'' = d_3p_3$, we see that $S = tp_2''+ sp_3''=0$, so that
$p''_2 = sq$ and $p''_3 = -tq$, hence
\[
I_U = \langle s,t \rangle \cap \langle p, q \rangle
\]
with $p,q$ both of bidegree $(1,1)$. But on $\PP$, $\V(p,q)$
is always nonempty, which would mean $I_U$ has a basepoint,
contradicting our hypothesis.
\end{proof}
\begin{remark}
If the $\P^1$ fibers of $Q$ did not intersect, Lemma~\ref{LS5-10} would follow
easily from Lemma~\ref{LS2-10}. However, because $Q$ is a projection
of $\Sigma_{3,1} \subseteq \P^7$ to $\P^5$, the $\P^1$ fibers of $Q$
do in fact intersect.
\end{remark}

\begin{thm}\label{AllLins}
If $U$ is basepoint free, then the only possibilities for linear first syzygies
of $I_U$ are
\begin{enumerate}
\item $I_U$ has a unique first syzygy of bidegree $(0,1)$ and no
other linear syzygies.
\item $I_U$ has a pair of first syzygies of bidegree $(0,1)$ and no
other linear syzygies.
\item $I_U$ has a unique first syzygy of bidegree $(1,0)$ and no
other linear syzygies.
\end{enumerate}
\end{thm}
\begin{proof}
It follows from Proposition~\ref{LS3} and Proposition~\ref{LS2} that
both of the first two items can occur. That there cannot be three
or more linear syzygies of bidegree $(0,1)$ follows easily from
the fact that if there are two syzygies of bidegree $(0,1)$ then
$I_U$ has the form of Proposition~\ref{LS2} and the resolution
is unique. Corollary~\ref{NoMix} shows there cannot be linear
syzygies of both bidegree $(1,0)$ and bidegree $(0,1)$ and  
Lemma~\ref{LS5-10} shows there can be at most one linear syzygy of
bidegree $(1,0)$.
\end{proof}

Our next theorem strengthens Lemma~\ref{LS4-10}: 
there is a minimal first syzygy of bidegree $(0,2)$ iff the $p$ in 
Lemma~\ref{LS2-10} factors. We need a pair of lemmas:
\begin{lem}\label{P2fiberBP}
If $\P(U)$ contains a $\P^2$ fiber of $\Sigma_{2,1}$, then $U$ is not 
basepoint free.
\end{lem}
\begin{proof}
If $\P(U)$ contains a $\P^2$ fiber of $\Sigma_{2,1}$ over a point of $\P^1$
corresponding to a linear form $l(u,v)$, after a change of 
basis $l(u,v)=u$ and so
\[
I_U = \langle s^2u,  stu, t^2u, l_1(s,t)l_2(s,t)v \rangle.
\]
This implies that $\langle u, l_1(s,t) \rangle \in \mbox{Ass}(I_U)$, so $U$ is not basepoint free.
\end{proof}
The next lemma is similar to a result of \cite{cds}, but differs due 
to the fact that the subspaces $\P(W) \subseteq \P(V)$ studied in \cite{cds} 
are always basepoint free.
\begin{lem}\label{02syz}
If $U$ is basepoint free, then there is a minimal first syzygy 
on $I_U$ of bidegree $(0,2)$ iff there 
exists $\P(W)\simeq \P^2 \subseteq \P(U)$ such that $\P(W) \cap \Sigma_{2,1}$ is a smooth conic.
\end{lem}
\begin{proof}
Suppose $\sum q_ip_i =0$ is a minimal first syzygy of bidegree $(0,2)$, 
so that $q_i = a_iu^2+b_iuv+c_iv^2$. Rewrite this as
$u^2 \sum a_ip_i +uv \sum b_ip_i +v^2 \sum c_ip_i = 0$ and  define 
$f_0 = \sum a_ip_i$, $f_1 = \sum b_ip_i$, $f_2 = \sum c_ip_i$.
By construction, $\langle f_0,f_1,f_2 \rangle \subseteq I_U$ and  
$[f_0,f_1,f_2]$ is a syzygy on $[u^2,uv,v^2]$, so
\[
\left[ \!
\begin{array}{c}
f_0\\
f_1\\
f_2
\end{array}\! \right] = \alpha \cdot \left[ \!
\begin{array}{c}
v\\
-u\\
0
\end{array}\! \right]+ \beta \cdot \left[ \!
\begin{array}{c}
0\\
v\\
-u
\end{array}\! \right] = \left[ \!
\begin{array}{c}
\alpha v\\
\beta v - \alpha u\\
-\beta u
\end{array}\! \right], \mbox{ for some }\alpha, \beta \in R_{2,0}.
\]
If $\{f_0,f_1,f_2\}$ are not linearly independent, there exist constants
$c_i$ with 
\[
c_0\alpha v +c_1(\beta v - \alpha u) -c_2 \beta u = 0.
\]
This implies that 
$(c_0 \alpha +c_1 \beta) v = (c_1\alpha -c_2\beta)u$, so 
$\alpha = k \beta$. But then $\{\alpha u, \alpha v \} \subseteq I_U$, which
means there is a minimal first syzygy of bidegree $(0,1)$, contradicting 
the classification of \S 2. Letting $W = \spn \{f_0,f_1,f_2\}$, we have
that $\P^2 \simeq \P(W) \subseteq \P(U)$. The actual bidegree $(0,2)$ 
syzygy is 
\[
\DET \left[ \!
\begin{array}{ccc}
v & 0 &f_0\\
-u & v& f_1\\
0  & -u &f_2
\end{array}\! \right] = 0. 
\]
To see that the $\P(W)$ meets $\Sigma_{2,1}$ in a smooth conic, note that 
by Lemma~\ref{P2fiberBP}, $\P(W)$ cannot be equal to a $\P^2$ 
fiber of $\Sigma_{2,1}$, or $\P(U)$ would
have basepoints. The image of the map $\P^1 \rightarrow \P(W)$ defined
by 
\[
(x:y) \mapsto x^2 (\alpha v) +xy(\beta v - \alpha u) + y^2(-\beta u)=(x\alpha +y \beta)(xv-yu)
\]
is a smooth conic $C \subseteq \P(W) \cap \Sigma_{2,1}$. Since 
$\P(W) \cap \Sigma_{2,1}$ is a curve of degree at most three, if
this is not the entire intersection, there would be a line $L$ residual 
to $C$. If $L \subseteq F_x$, where $F_x$ is a $\P^2$ fiber over $x \in \P^1$, then 
for small $\epsilon$, $F_{x+\epsilon}$ also meets $\P(W)$ in a line,
which is impossible. If $L$ is a $\P^1$ fiber of $\Sigma_{2,1}$, this
would result in a bidegree $(0,1)$ syzygy, which is impossible 
by the classification of \S 2.
\end{proof}
\begin{defn} A line $l \subseteq \P(s^2,st,t^2)$ with $l=af+bg$, 
$f,g \in \spn \{s^2,st,t^2\}$ is split if $l$ has a fixed factor: 
for all $a,b \in \P^1$, $l = L(aL'+bL'')$ with $L \in R_{1,0}$. 
\end{defn}
\begin{thm}\label{Type4resoln}
If $U$ is basepoint free, then $I_U$ has minimal first syzygies 
of bidegree $(1,0)$ and $(0,2)$ iff 
\[
\P(U) \cap \Sigma_{2,1} = C \cup L,
\]
where $L\simeq \P(W')$ is a split line in 
a $\P^2$ fiber of $\Sigma_{2,1}$ and  $C$ is
a smooth conic in $\P(W)$, such that 
$\P(W') \cap \P(W) = C \cap L$ is a point and  $\P(W') + \P(W) = \P(U)$.
\end{thm}
\begin{proof}
Suppose there are minimal first syzygies of bidegrees $(1,0)$ and $(0,2)$.
By Lemma~\ref{02syz}, the $(0,2)$ syzygy determines a conic $C$ in a 
distinguished $\P(W) \subseteq \P(U)$. Every point of $C$ lies on both a $\P^2$
and $\P^1$ fiber of $\Sigma_{2,1}$. No $\P^1$ fiber of $\Sigma_{2,1}$ is
contained in $\P(U)$, or there would be a first syzygy of 
bidegree $(0,1)$, which is impossible by Corollary~\ref{NoMix}. 
By Lemma~\ref{LS2-10}, there exists 
$W' = \spn \{ps,pt \} \subseteq U$, so we have a distinguished line
$\P(W') \subseteq \P(U)$. We now consider two possibilities:
\vskip .04in
\noindent Case 1: If $p$ factors, then $\P(W')$ is a split line contained in $\Sigma_{2,1}$, 
which must therefore be contained in a $\P^2$ fiber and  $p = L(s,t)l(u,v)$, 
where $l(u,v)$ corresponds to a point of a $\P^1$ fiber of $\Sigma_{2,1}$ and  
$\P(W') \cap \P(W)$ is a point. In particular 
$\P(U) \cap \Sigma_{2,1}$ is the union of a line and conic, 
which meet transversally at a point. 
\vskip .04in
\noindent Case 2: If $p$ does not factor, then $p=a_0su+a_1sv+a_2tu+a_3tv$, 
$a_0a_3-a_1a_2 \ne 0$. The corresponding line $L = \P(ps,pt)$ 
meets $\P(W)$ in a point and since $p$ is irreducible $L \cap C = \emptyset$. 
Since $W = \spn \{\alpha v, \beta v - \alpha u, -\beta v \}$, we must have 
\[
p \cdot L_0(s,t) = a \alpha v -b \beta u + c(\beta v - \alpha u)
\]
for some $\{a,b,c\}$, where $L_0(s,t)=b_0s+b_1t$ corresponds to $L \cap \P(W)$.
Write $p \cdot L_0(s,t)$ as $l_1uL_0(s,t) + l_2vL_0(s,t)$, where 
\[
\begin{array}{ccc}
l_1(s,t) &= &a_0s+a_2t\\
l_2(s,t) &= &a_1s+a_3t
\end{array}
\]
Then 
\[
\begin{array}{ccc}
p \cdot L_0(s,t) & = & l_1(s,t)L_0(s,t)u + l_2(s,t)L_0(s,t)v \\
                 & = & a \alpha v -b \beta u + c(\beta v - \alpha u)\\
                 & = & (-b\beta -c\alpha)u +(a \alpha +c \beta) v
\end{array}
\]
In particular, 
\begin{equation}\label{lMat}
\left[ \!
\begin{array}{cc}
a & c\\
c & b
\end{array}\! \right] \cdot \left[ \!
\begin{array}{c}
\alpha\\
\beta
\end{array}\! \right] = \left[ \!
\begin{array}{c}
l_2 L_0\\
-l_1 L_0
\end{array}\! \right] 
\end{equation}
If \[
\DET \left[ \!
\begin{array}{cc}
a & c\\
c & b
\end{array}\! \right] \ne 0, 
\]
then applying Cramer's rule to Equation~\ref{lMat} 
shows that $\alpha$ and $\beta$ share
a common factor $L_0(s,t)$. But then 
$W = \spn \{L_0 \gamma, L_0 \delta, L_0 \epsilon \}$, which
contradicts the basepoint freeness of $U$: change coordinates
so $L_0 =s$, so $U = s\gamma, s \delta, s\epsilon, p_3$. Since
$p_3|_{s=0} = t^2l(u,v)$, $\langle s, l(u,v) \rangle$ is an 
associated prime of $I_U$, a contradiction. To conclude,
consider the case $ab-c^2 =0$. Then there is a constant $k$ 
such that 
\[
\left[ \!
\begin{array}{cc}
a & c\\
ka & kc
\end{array}\! \right] \cdot \left[ \!
\begin{array}{c}
\alpha\\
\beta
\end{array}\! \right] = \left[ \!
\begin{array}{c}
l_2 L_0\\
-l_1 L_0
\end{array}\! \right], 
\]
which forces $kl_1(s,t) = -l_2(s,t)$. Recalling that $l_1(s,t) = a_0s+a_2t$ and $l_2(s,t) = a_1s+a_3t$, this implies that 
\[
\DET \left[ \!
\begin{array}{cc}
a_0 & a_2\\
a_1 & a_3
\end{array}\! \right] = 0, 
\]
contradicting the irreducibility of $p$.\newline

This shows that if $I_U$ has minimal first syzygies of bidegree $(1,0)$ and
$(0,2)$, then $\P(U) \cap \Sigma_{2,1} = C \cup L$ meeting transversally
at a point. The remaining implication follows from Lemma~\ref{LS4-10} and
Lemma~\ref{02syz}.
\end{proof}

For $W \subseteq H^0(\OPP(2,1))$ a basepoint free subspace of 
dimension three, the minimal free resolution of $I_W$ is 
determined in \cite{cds}: there are two possible  minimal 
free resolutions, which depend only whether 
$\P(W)$ meets $\Sigma_{2,1}$ in a finite set of points, or a smooth
conic  $C$. By Theorem~\ref{Type4resoln}, if there are minimal
first syzygies of bidegrees $(1,0)$ and $(0,2)$, then $U$ contains
a $W$ with $\P(W) \cap \Sigma_{2,1} = C$, which suggests building 
the Type 4 resolution by choosing $W = \spn \{p_0,p_1,p_2\}$ to 
satisfy $\P(W) \cap \Sigma_{2,1} = C$ and  constructing a mapping
cone. There are two problems with this approach. First, there
does not seem to be an easy description for $\langle p_0,p_1,p_2\rangle :p_3$.
Second, recall that the mapping cone resolution need not be minimal.
A computation shows that the shifts in the resolutions of 
$\langle p_0,p_1,p_2\rangle :p_3$ and $\langle p_0,p_1,p_2\rangle$ 
overlap and  there are many cancellations. However, choosing $W$ to 
consist of two points on $L$ and one on $C$ solves both of these 
problems at once.
\begin{lem}\label{Type4L1}
In the setting of Theorem~\ref{Type4resoln}, let $p_0$ correspond to
$L \cap C$ and  let $p_1, p_2$ correspond to points on $L$ and $C$
(respectively) distinct from $p_0$. 
If $W = \spn \{p_0,p_1,p_2 \}$, then $I_W$ has a Hilbert
Burch resolution. Choosing coordinates so 
$W = \spn \{sLv, tLv, \beta u \}$, the primary decomposition of $I_W$ is 
\[
\cap_{i=1}^4 I_i = \langle s,t\rangle^2 \cap \langle u,v \rangle \cap \langle \beta, v \rangle \cap \langle L, u \rangle.
\]
\end{lem}
\begin{proof}
After a suitable change of coordinates, $W = \spn \{sLv, tLv, \beta u \}$,
and writing $\beta = (a_0s+a_1t)L'(s,t)$, $I_W$ consists of the two by 
two minors of 
\[
\phi_2 = 
\left[ \!
\begin{array}{cc}
t & a_0uL'\\
-s & a_1uL' \\
0 &L v
\end{array}\! \right].
\]
Hence, the minimal free resolution of $I_W$ is
\[
0 \longleftarrow I_W \longleftarrow (-2,-1)^3 \stackrel{\phi_2}{\longleftarrow} 
 (-3,-1) \oplus (-3,-2) \longleftarrow 0
\]
For the primary decomposition, since $L$ does not divide $\beta$, 
\[
\langle \beta, v \rangle \cap \langle L, u \rangle = \langle \beta L, vL, \beta u, vu \rangle,
\]
and intersecting this with $\langle s,t\rangle^2 \cap \langle u,v \rangle$
gives $I_W$.
\end{proof}

\begin{lem}\label{Type4L2}
If $U$ is basepoint free, $W = \spn \{sLv, tLv, \beta u \}$ and 
$p_3 = \alpha u - \beta v = tL u - \beta v$, then
\[
I_W : p_3 = \langle \beta L, vL, \beta u, vu \rangle.
\]
\end{lem}
\begin{proof}
First, our choice of $p_0$ to correspond to $C \cap L$ in 
Theorem~\ref{Type4resoln} means we may write $\alpha = tL$.
Since $\langle s,t \rangle^2: p_3 = 1 = \langle u,v \rangle : p_3$, 
\[
I_W : p_3 =   (\cap_{i=1}^4 I_i):p_3 = (\langle \beta, v \rangle:p_3) \cap 
(\langle L, u \rangle : p_3).
\]
Since $p_3 = tL u - \beta v$, $fp_3 \in \langle \beta, v \rangle$ iff 
$f tLu \in \langle \beta, v \rangle$. Since $tL=\alpha$ and 
$\alpha, \beta$ and $u,v$ are relatively prime, this implies $f \in \langle \beta, v \rangle$. The same argument shows that $\langle L, u \rangle : p_3$ must equal $\langle L, u \rangle$.
\end{proof}

\begin{thm}\label{Type4MC}
In the situation of Theorem~\ref{Type4resoln}, the minimal 
free resolution is of Type 4. If $p_0$ corresponds to
$L \cap C$, $p_1\ne p_0$ to another point on $L$ and  $p_2 \ne p_0$ 
to a point on $C$ and $W = \spn \{p_0,p_1,p_2 \}$, then the minimal
free resolution is given by the mapping cone of $I_W$ and $I_W : p_3$.
\end{thm}
\begin{proof}
We construct a mapping cone resolution from the short exact sequence
\[
0 \longleftarrow R/I_U  \longleftarrow   R/I_W  \stackrel{\cdot p_3} \longleftarrow R(-2,-1)/I_W:p_3 \longleftarrow 0.
\]
By Lemma~\ref{Type4L2}, 
\[
I_W : p_3 = \langle \beta L, Lv, \beta u, uv \rangle,
\]
which by the reasoning in the proof of Proposition~\ref{LS2} 
has minimal free resolution:
\begin{small}
\[
I_W : p_3 \longleftarrow 
\begin{array}{c}
(-3,-0)\\
\oplus \\
(-1,-1)\\
\oplus \\
(-2,-1)\\
\oplus \\
(0,-2)
\end{array}\! \xleftarrow{\left[ \!
\begin{array}{cccc}
v      & 0      &u  &0 \\
-\beta & 0      &0  &u \\
0      & v      &-L &0   \\
0      & -\beta &0  & -L
\end{array}\! \right]} 
\begin{array}{c}
(-3,-1)\\
\oplus \\
(-2,-2)\\
\oplus \\
(-3,-1)\\
\oplus \\
(-1,-2)
\end{array}\!
\xleftarrow{\left[ \!\begin{array}{c}
u\\
-L\\
-v \\
\beta
\end{array}\! \right]}
 (-3,-2).
\]
\end{small}
A check shows that there are no overlaps in the mapping cone shifts,
hence the mapping cone resolution is actually minimal.
\end{proof}

\subsection{Type 3 resolution}
Finally, suppose $I_U=\langle p_0, p_1, p_2, p_3 \rangle$ with 
$p_0= ps$ and  $p_1=pt$, such that $p=a_0su+a_1sv+a_2tu+a_3tv$ is 
irreducible, so $a_0a_3-a_1a_2 \not = 0$. 
As in the case of Type 4, the
minimal free resolution will be given by a mapping cone. However, in Type 3
the construction is more complicated: we will need two mapping cones to 
compute the resolution. What is surprising is that by a judicious 
change of coordinates, the bigrading allows us to reduce $I_U$ so that
the equations have a very simple form. 
\begin{thm}\label{T3res}
If $U$ is basepoint free and $I_U =\langle ps, pt, p_2, p_3 \rangle$ 
with $p$ irreducible, then the $I_U$ has a mapping cone resolution,
and is of numerical Type 3. 
\end{thm}
\begin{proof}
Without loss of generality, 
assume $a_0  = 1$.  Reducing $p_2$ and $p_3$ mod  $ps$ and $pt$, 
we have 
$$
\begin{array}{ccc}
p_2 & = & b_0t^2u+b_1s^2v+b_2stv+b_3t^2v\\
p_3 & = & c_0t^2u+c_1s^2v+c_2stv+c_3t^2v
\end{array}
$$
Since $U$ is basepoint free, either $b_0$ or $c_0$ is nonzero, so 
after rescaling and reducing $p_3$ mod $p_2$ 
\[
\begin{array}{ccr}
p_2 & = & t^2u+b_1s^2v+b_2stv+b_3t^2v\\
p_3 & = & c_1s^2v+c_2stv+c_3t^2v \\
& = & (c_1s^2+c_2st+c_3t^2)v\\
& = & L_1L_2v\\
\end{array}
\]
for some $L_i \in R_{1,0}$. If the $L_i$'s are linearly independent,
then a change of variable replaces $L_1$ and $L_2$ with $s$ and
$t$. This transforms $p_0, p_1$ to $p'l_1, p'l_2$, but since the $l_i$'s are 
linearly independent linear forms in $s$ and $t$, 
$\langle p'l_1, p'l_2 \rangle = \langle p's, p't \rangle$, with
$p'$ irreducible. So we may assume
\[I_U=\langle ps, pt, p_2, p_3 \rangle, \mbox{ where } p_3 = stv \mbox{ or } s^2v.
\]
With this change of variables,
\[
p_2  =  l^2u+(b'_1s^2+b'_2st+b'_3t^2)v
\]
where $l=as+bt$ and $b \ne 0$. We now analyze the two
possible situations. First, suppose $p_3 = stv$.
Reducing $p_2$ modulo $\langle ps, pt,stv \rangle$  
yields
\[
\begin{array}{ccr}
p_2  & = &  \alpha t^2u+b''_1s^2v+b''_3t^2v\\
  & =& (\alpha u+b''_3v)t^2+b''_1s^2v
\end{array}
\]
By basepoint freeness, $\alpha \not = 0$, so changing variables 
via $\alpha u+b''_3v \mapsto u$ yields $p_2 = t^2u+b''_1s^2v$. 
Notice this change of variables does not change the form of the 
other $p_i$. Now $b_1'' \ne 0$ by basepoint freeness, so 
rescaling $t$ (which again preserves the form of the other $p_i$)
shows that
\[
I_U=\langle ps, pt, t^2u+s^2v, stv \rangle.
\]
Since $p$ is irreducible, $p=sl_1+tl_2$ with $l_i$ are linearly 
independent elements of $R_{0,1}$. Changing variables once again via
$l_1 \mapsto u$ and $l_2 \mapsto v$, we have
\[
I_U=\langle s(su+tv), t(su+tv), stQ_1, s^2Q_1+ t^2Q_2 \rangle,
\]
where $Q_1=au+bv, Q_2=cu+dv$ are linearly independent with $b \ne 0$. 
Rescaling $v$ and $s$ we may assume $b=1$. Now let
\[
I_W=\langle s(su+tv), t(su+tv), stQ_1 \rangle.
\]
The minimal free resolution of $I_W$ is
\[
0 \longleftarrow I_W \longleftarrow (-2,-1)^3 \stackrel{\begin{small}\left[ \!
\begin{array}{cc}
t & 0\\
-s & sQ_1 \\
0 &p
\end{array}\! \right]\end{small}}{\longleftarrow} 
 (-3,-1) \oplus (-3,-2) \longleftarrow 0,
\]
To obtain a mapping cone resolution, we need to compute $I_W:p_2$. 
As in the Type 4 setting, we first find the primary decomposition for $I_W$.
\begin{enumerate}
\item   If $a = 0$, then 
$$
\begin{array}{cll}
I_W&=& \langle s(su+tv), t(su+tv), stv \rangle\\
&=&\langle u, v \rangle \cap \langle s, t \rangle^2 \cap \langle u, t \rangle \cap \langle su+tv, (s, v)^2 \rangle = \cap_{i=1}^4 I_i
\end{array}
$$
\item  If  $a \ne 0$, rescale $u$ and $t$ by $a$ so $Q_1=u+v$. Then
\[
\begin{array}{cll}
I_W&=& \langle s(su+tv), t(su+tv), st(u+v) \rangle\\
&=&\langle u, v \rangle \cap \langle s, t \rangle^2 \cap \langle v, s \rangle \cap \langle u, t \rangle \cap \langle u+v, s-t \rangle = \cap_{i=1}^5 I_i
\end{array}
\]
\end{enumerate}
Since $s^2Q_1+ t^2Q_2 \in I_1 \cap I_2$ in both cases, 
$I_W: s^2Q_1+ t^2Q_2 = \cap_{i=2}^n I_i$. So if $a=0$, $I_W: s^2Q_1+ t^2Q_2 $
\[
\begin{array}{ccl}
& = &  \langle u,t \rangle : s^2Q_1+ t^2Q_2 \ \cap \langle su+tv, (s, v)^2 \rangle : s^2Q_1+ t^2Q_2 \\
 &=&  \langle u,t \rangle \cap \langle su+tv, (s, v)^2 \rangle\\
 &=&\langle su+tv, uv^2, tv^2, stv, s^2t \rangle\\
 &=& \langle su+tv, I_3(\phi)  \rangle,
 \end{array}
 \]
while if $a \ne 0$, $I_W: s^2Q_1+ t^2Q_2$ 
\[
 \begin{array}{ccl}
& = &  \langle v,s \rangle : s^2Q_1+ t^2Q_2  \ \cap \langle u,t \rangle: s^2Q_1+ t^2Q_2 \ \cap   \langle u+v, s-t \rangle : s^2Q_1+ t^2Q_2\\
  &=&  \langle v, s \rangle \cap \langle u, t \rangle \cap \langle u+v, s-t \rangle \\
 &=&\langle su+tv, uv(u+v), tv(u+v), tv(s-t), st(s-t) \rangle\\
 &=& \langle su+tv, I_3(\phi)  \rangle
 \end{array}
\]
 where 
\[
\phi=\left[ \begin{array}{lll}  t & 0 & 0\\ u & s & 0 \\ 0 & v & s \\ 0 & 0 & v \end{array}\right] \mbox{ if } a=0, \mbox{ and } 
\phi=\left[ \begin{array}{lll}  t & 0 & 0\\ u & s-t & 0 \\ 0 & u+v & s \\ 0 & 0 & v \end{array}\right] \mbox{ if } a \ne 0 . 
\]
Since $I_3(\phi)$ has a Hilbert-Burch resolution, a resolution 
of $I_W: p_2 = \langle su+tv, I_3(\phi)\rangle$ can be obtained as 
the mapping cone of $I_3(\phi)$ with $I_3(\phi): p$. There are no 
overlaps, so the result is a minimal resolution. However, there is 
no need to do this, because the change of variables allows us to 
do the computation directly and  we find
\[
0 \leftarrow I_w:p_3 \longleftarrow \begin{array}{c}
(-1,-1)\\
 \oplus \\
 (0,-3)\\
 \oplus \\
(-1,-2)\\
\oplus \\
(-2,-1)\\
\oplus \\
(-3,-0)\\
\end{array} \longleftarrow 
\begin{array}{c}
 (-1,-3)^2\\
 \oplus \\
(-2,-2)^2\\
\oplus \\
(-3,-1)^2\\
 \end{array} \longleftarrow 
\begin{array}{c}
 (-2,-3)\\
 \oplus \\
(-3,-2)\\
 \end{array} \leftarrow 0 
\]
This concludes the proof if $p_3 = stv$. When $p_3 = s^2v$, the argument 
proceeds in similar, but simpler, fashion.
\end{proof}
 
\section{No linear first syzygies}
\subsection{Hilbert function}

\begin{prop}\label{hilb}
If $U$ is basepoint free, then there are six types of bigraded Hilbert function in one-to-one correspondence with the resolutions of Table 2. The tables below contain the values of $h_{i,j}=HF((i,j),R/I_U)$, for $i<5,j<6$ listed in the order corresponding to the six numerical types in Table~\ref{T2}. 

\begin{small}
$$
\begin{matrix}

\begin{array}{c|ccccc}
   & 0 & 1 & 2 & 3 & 4 \\
   \hline
  0 & 1 & 2 & 3 & 4 & 5 \\
1 & 2 & 4 & 6 & 8 & 10 \\
2 & 3 & 2 & 1 & 0 & 0 \\
3 & 4 & 0 & 0 & 0 &0 \\
4 & 5 & 0 & 0 & 0 &0 \\
5 & 6 & 0 & 0 & 0 &0 \\
\end{array}
&
\begin{array}{c|ccccc}
   & 0 & 1 & 2 & 3 & 4 \\
   \hline
  0 & 1 & 2 & 3 & 4 & 5 \\
1 & 2 & 4 & 6 & 8 & 10 \\
2 & 3 & 2 & 1 & 1 & 1 \\
3 & 4 & 0 & 0 & 0 &0 \\
4 & 5 & 0 & 0 & 0 &0 \\
5 & 6 & 0 & 0 & 0 &0 \\
\end{array}
&
\begin{array}{c|ccccc}
   & 0 & 1 & 2 & 3 & 4 \\
   \hline
  0 & 1 & 2 & 3 & 4 & 5 \\
1 & 2 & 4 & 6 & 8 & 10 \\
2 & 3 & 2 & 1 & 0 & 0 \\
3 & 4 & 1 & 0 & 0 &0 \\
4 & 5 & 0 & 0 & 0 &0 \\
5 & 6 & 0 & 0 & 0 &0 \\
\end{array}
\\ \\
\begin{array}{c|ccccc}
   & 0 & 1 & 2 & 3 & 4 \\
   \hline
  0 & 1 & 2 & 3 & 4 & 5 \\
1 & 2 & 4 & 6 & 8 & 10 \\
2 & 3 & 2 & 1 & 1 & 1 \\
3 & 4 & 1 & 0 & 0 &0 \\
4 & 5 & 0 & 0 & 0 &0 \\
5 & 6 & 0 & 0 & 0 &0 \\
\end{array}
&
\begin{array}{c|ccccc}
   & 0 & 1 & 2 & 3 & 4 \\
   \hline
  0 & 1 & 2 & 3 & 4 & 5 \\
1 & 2 & 4 & 6 & 8 & 10 \\
2 & 3 & 2 & 2 & 2 & 2 \\
3 & 4 & 0 & 0 & 0 &0 \\
4 & 5 & 0 & 0 & 0 &0 \\
5 & 6 & 0 & 0 & 0 &0 \\
\end{array}
&
\begin{array}{c|ccccc}
   & 0 & 1 & 2 & 3 & 4 \\
   \hline
  0 & 1 & 2 & 3 & 4 & 5 \\
1 & 2 & 4 & 6 & 8 & 10 \\
2 & 3 & 2 & 3 & 4 & 5 \\
3 & 4 & 0 & 0 & 0 &0 \\
4 & 5 & 0 & 0 & 0 &0 \\
5 & 6 & 0 & 0 & 0 &0 \\
\end{array}
\end{matrix}
$$
\end{small}
\end{prop}

The entries of the first two rows and the first column are clear: $$h_{0,j}=HF((0,j),R)=j+1,h_{1,j}=HF((1,j),R)=2j+2$$ and $h_{i,0}=HF((i,0),R)=i+1$. Furthermore $h_{2,1}=2$ by the linear independence of the minimal generators of $I_U$. The proof of the proposition is based on the following lemmas concerning  $h_{i,j}, i\geq 2, j\geq 1$.

\begin{lem} \label{ABlemma}
For $I_U$ an ideal generated by four independent forms of bidegree $(2,1)$
\begin{enumerate}
\item $h_{3,1}$ is the number of bidegree $(1,0)$ first syzygies
\item $h_{2,2}-1$ is the number of bidegree $(0,1)$ first syzygies
\end{enumerate}
\end{lem}
\begin{proof}
From the free resolution 
$$0 \leftarrow R/I_U \leftarrow R \leftarrow R(-2,-1)^4 \longleftarrow \begin{array}{c}
 R(-2,-3)^B\\
 \oplus \\
R(-3,-1)^A\\
 \oplus \\
F_1
\end{array} \longleftarrow F_2 \longleftarrow F_3\leftarrow 0,$$
we find
$$h_{3,1}=HF((3,1),R)-4HF((1,0),R)+AHF((0,0),R)=8-8+A=A$$
$$h_{2,2}=HF((2,2),R)-4HF((0,1),R)+BHF((0,0),R)=9-8+B=B+1,$$
since $F_i$ are free $R$-modules generated in 
degree $(i,j)$ with $i>3$ or $j>2$. \end{proof}

\begin{lem}
If $U$ is basepoint free, then  $h_{3,2}=0$ for every numerical type.
 \end{lem}
 
 \begin{proof}
 If there are no bidegree $(1,0)$ syzygies then $HF((3,1),R/I_U)=0$ and consequently $HF((3,2),R/I_U)=0$. 
 If there are bidegree  $(1,0)$ syzygies then we are in Type 3 or 4 where by Proposition \ref{LS3-10} we know the relevant part of the resolution is
$$0 \leftarrow R/I_U \leftarrow R \leftarrow R(-2,-1)^4 \longleftarrow \begin{array}{c}
R(-3,-1)\\
 \oplus \\
  R(-3,-2)^2\\
 \oplus \\
F_1
\end{array} \longleftarrow F_2 \longleftarrow F_3\leftarrow 0$$
Then 
$$\begin{array}{ccl}
h_{3,2} &= &HF((3,2),R)-4HF((1,1),R)+HF((0,1),R)+2HF((0,0),R) \\
 & = &12-16+2+2=0.
 \end{array}
 $$
 \end{proof}
 
 So far we have determined the following shape of the Hilbert function of $R/I_U$:
 \begin{small}
 $$\begin{array}{c|cccccc}
   & 0 & 1 & 2 & 3 & 4 \\
   \hline
  0 & 1 & 2 & 3 & 4 & 5 \\
1 & 2 & 4 & 6 & 8 & 10 \\
2 & 3 & 2 & h_{2,2}& h_{2,3} & h_{2,4} \\
3 & 4 & h_{3,1} & 0 & 0 &0 \\
4 & 5 & h_{4,1} & 0 & 0 &0 \\
5 & 6 & h_{5,1} & 0 & 0 &0 \\
\end{array}$$
\end{small}
If linear syzygies are present we know from the previous section the exact description of the possible minimal resolutions of $I_U$ and it is an easy check that they agree with the last four Hilbert functions in Proposition \ref{hilb}. Next we focus on the case when no linear syzygies are present. By 
Lemma \ref{ABlemma} this yields $h_{2,2}=1$ and $h_{3,1}=0$, hence $h_{i,1}=0$ for $i\geq 3$. We show that in the absence of linear syzygies only the first two Hilbert functions in Proposition \ref{hilb} may occur:

\subsection{Types 1 and 2}
In the following we assume that the basepoint free ideal $I_U$ has no linear syzygies. We first determine the maximal numerical types which correspond to the  Hilbert functions found in \S 4.1 and  then we show that only the Betti 
numbers corresponding to linear syzygies cancel. 

\begin{prop}\label{HFnoLin}
If $U$ is basepoint free and $I_U$ has no linear syzygies, then
\begin{enumerate}
\item $I_U$ cannot have two or more linearly independent bidegree $(0,2)$ first syzygies
\item $I_U$ cannot have two minimal first syzygies of bidegrees $(0,2)$, $(0,j)$, $j> 2$ 
\item $I_U$ has a single bidegree $(0,2)$ minimal syzygy  iff $h_{2,j}=1$ for $j\geq 3$
\item $I_U$ has no bidegree $(0,2)$ minimal syzygy iff $h_{2,j}=0$ for $j\geq 3$
\end{enumerate}
\end{prop}

\begin{proof}

(1) Suppose $I_U$ has two  linearly independent bidegree $(0,2)$ first syzygies which can be written down by a similar procedure to the one used in  Lemma \ref{LS1} as
$$\begin{matrix}
u^2p+uvq+v^2r & =& 0\\
u^2p'+uvq'+v^2r' & =& 0
\end{matrix}$$ 
with $p,q,r,p',q',r' \in U$. Write $p=p_1u+p_2v$ with $p_1,p_2\in R_{2,0}$ and similarly for $p',q,q',r,r'$. Substituting in the equations above one obtains
$$\begin{matrix}
p_1 & =& 0 & & p'_1& = & 0\\
p_2+q_1 & =& 0 & & p'_2 +q_1'& = & 0\\
q_2+r_1 & =& 0 & & q'_2+r'_1 & = & 0\\
r_2 & =& 0 & & r'_2 & = & 0\\
\end{matrix}$$
hence 
\[
\begin{array}{cccccc}
p & =&p_2v,\mbox{ } & p' &= &p'_2v\\
q & =& -(p_2u+r_1v), \mbox{ } & q'&=&-(p'_2u+r'_1v)\\
r &= &r_1v,  \mbox{ } & r'&=&r'_1v
\end{array}
\]
are elements of $I_U$. If both of the pairs $p_2v, p'_2v$ or $r_1u, r'_1u$ consists of linearly independent elements of $R_{2,1}$, then $U\cap \Sigma_{2,1}$ contains a $\P^1$ inside each of the $\P^2$ fibers over the points corresponding to $u,v$ in the $\P^1$ factor of the map $\P^2\times \P^1\mapsto \P^5$. Pulling back the two lines from $\Sigma_{2,1}$ to the domain of its defining map, one obtains two lines in $\P^2$ which must meet (or be identical). Taking the image of the intersection point we get two elements of the form $\alpha u, \alpha v \in I_U$ which yield a $(0,1)$ syzygy, thus contradicting our assumption. Therefore it must be the case that $p_2'=ap_2$ or $r_1'=br_1$ with $a,b\in k$. The reasoning being identical, we shall only analyze the case $p_2'=ap_2$. A linear combination of the elements  $q=-(p_2u+r_1v), q'=-(p'_2u+r'_1v)\in I_U$ produces $(r'_1-ar_1)v\in I_U$ and a linear combination of the elements  $r_1u, r'_1u\in I_U$ produces $(r'_1-ar_1)u\in I_U$, hence again we obtain a $(0,1)$ syzygy unless $r'_1=ar_1$. But then $(p',q',r')=a(p,q,r)$ and these triples yield linearly dependent bidegree $(0,2)$ syzygies.

(2) The assertion that $I_U$ cannot have a bidegree $(0,2)$ and a (distinct) bidegree $(0,j)$, $j\geq 2$ minimal first syzygies is proved by induction on $j$. The base case $j=2$ has already been solved. Assume $I_U$ has a 
degree $(0,2)$ syzygy $u^2p+uvq+v^2r =0$ with $p=p_1u+p_2v,q,r\in I_U$ and a  bidegree $(0,j)$ syzygy $u^jw_1+u^{j-1}vw_2+\ldots +v^jw_{j+1}=0$ with $w_i=y_iu+z_iv\in I_U$. Then as before $p_1=0, z_1=0,r_2=0,z_{j+1}=0$ and the same reasoning shows one must have $z_1=ap_2$ or  $y_{j+1}=br_1$. Again we handle the case $z_1=ap_2$ where a linear combination of the two syzygies produces the new syzygy 
\[
u^{j-1}v(w_2-aq)+u^{j-2}v^2(w_3-ar)+u^{j-3}v^3(w_3)\ldots +v^jw_{j+1}=0.
\]
Dividing by $v$: $u^{j-1}(w_2-aq)+u^{j-2}v(w_3-ar)+u^{j-3}v^2(w_3)\ldots +v^{j-1}w_{j+1}=0$, which is a minimal bidegree $(0,j-1)$ syzygy iff the original $(0,j)$ syzygy was minimal. This contradicts the induction hypothesis.

(3) An argument similar to Lemma \ref{ABlemma} shows that in the absence of $(0,1)$ syzygies $h_{2,3}$ is equal to the number of bidegree $(0,2)$ syzygies on $I_U$. Note that the absence of $(0,1)$ syzygies implies there can be no bidegree $(0,2)$ second syzygies of $I_U$ to cancel the effect of bidegree $(0,2)$ first syzygies on the Hilbert function. This covers the converse implications of both (3) and (4) as well as the case $j=3$ of the direct implications. The computation of $h_{2,j}$, $j\geq 3$ is completed as follows
$$\begin{matrix}
h_{2,j} &= & HF((2,j),R)-HF((2,j),R(-2,-1)^4)+HF((2,j),R(-2,-3))\\
& = &3(j+1)-4j+(j-2)\\
& =1
\end{matrix}
$$

(4) In this case we compute 
$$h_{2,3}=HF((2,3),R)-HF((2,3),R(-2,-1)^4)=12-12=0$$
$$HF((2,4),R)-HF((2,4),R(-2,-1)^4)=15-16=-1$$
The fact that $h_{2,j}=0$ for higher values of $j$ follows from $h_{2,3}=0$. In fact even more is true: $I_U$ is forced to have a single bidegree $(0,3)$ first syzygy to ensure that $h_{2,j}=0$  for $j\geq 4$.
\end{proof}

\begin{cor}
There are only two possible Hilbert functions for basepoint free ideals $I_U$ 
without linear syzygies, depending on whether there is no $(0,2)$ syzygy or exactly one $(0,2)$ syzygy. The  two possible Hilbert functions are 
\begin{small}
$$
\begin{matrix}
\begin{array}{c|ccccc}
   & 0 & 1 & 2 & 3 & 4 \\
   \hline
  0 & 1 & 2 & 3 & 4 & 5 \\
1 & 2 & 4 & 6 & 8 & 10 \\
2 & 3 & 2 & 1 & 0 & 0 \\
3 & 4 & 0 & 0 & 0 &0 \\
4 & 5 & 0 & 0 & 0 &0 \\
5 & 6 & 0 & 0 & 0 &0 \\
\end{array}
& & &
\begin{array}{c|ccccc}
   & 0 & 1 & 2 & 3 & 4 \\
   \hline
  0 & 1 & 2 & 3 & 4 & 5 \\
1 & 2 & 4 & 6 & 8 & 10 \\
2 & 3 & 2 & 1 & 1 & 1 \\
3 & 4 & 0 & 0 & 0 &0 \\
4 & 5 & 0 & 0 & 0 &0 \\
5 & 6 & 0 & 0 & 0 &0 \\
\end{array}
\end{matrix}
$$
\end{small}
\end{cor}

 \begin{prop}\label{syznoLin}
If $U$ is basepoint free and $I_U$ has no linear syzygies, then $I_U$ has 
\begin{enumerate}
\item exactly 4 bidegree $(1,1)$ first syzygies 
\item exactly 2 bidegree $(2,0)$ first syzygies
\end{enumerate}
\end{prop}

\begin{proof}
Note that there cannot be any second syzygies in bidegrees $(1,1)$ and $(2,0)$ because of the absence of linear first syzygies. Thus the numbers $\beta^1_{3,2}, \beta^1_{4,1}$ of bidegree $(1,1)$ and $(2,0)$ first syzygies are determined by the Hilbert function:
$$\beta^1_{3,2}=h_{3,2}-HF((3,2),R)+HF((3,2),R(-2,-1)^4)=0-12+16=4$$
$$\beta^1_{4,1}=h_{4,1}-HF((4,1),R)+HF((4,1),R(-2,-1)^4)=0-10+12=2$$
\end{proof}

Next we obtain upper bounds on the bigraded Betti numbers of $I_U$ by using bigraded initial ideals. The concept of  initial ideal with respect to any fixed term order is well known and so is the cancellation principle asserting that the resolution of an ideal can be obtained from that of its initial ideal by cancellation of some consecutive syzygies of the same bidegree. In general the problem of determining which cancellations occur is very difficult. In the following we exploit the cancellation principle by using the bigraded setting to our advantage. For the initial ideal computations we use the revlex order induced  by $s>t>u>v$.  

In \cite{ACD}, Aramova, Crona and de Negri introduce bigeneric initial ideals as follows (we adapt the definition to our setting): let $G={\bf{GL}}(2,2) \times {\bf{GL}}(2,2)$ with an element  $g=(d_{ij}, e_{kl}) \in G$
acting on the variables in $R$ by
$$
g: s \mapsto d_{11}s+d_{12}t , \ t \mapsto d_{21}s+d_{22}t ,\ u \mapsto e_{11}u+e_{12}v , \ v \mapsto e_{21}u+e_{22}v\,\,\,
$$
We shall make use of the following results of \cite{ACD}.
\begin{thm}\label{thmacd}$[$\cite{ACD} Theorem 1.4$]$
Let $I \subset R$ be a bigraded ideal. There is a Zariski open set $U$ in  $G$ and an ideal $J$ such that for all $g\in U$ we have $ in(g(I))=J$.
\end {thm}

\begin{defn}The ideal $J$ in Theorem~\ref{thmacd} is defined to be the bigeneric initial ideal of $I$, denoted by  $bigin(I)$.
\end{defn}

\begin{defn}
A monomial ideal $I\subset R$ is  bi-Borel fixed if $g(I)=I$ for any upper triangular matrix $g\in G$.
\end{defn}

\begin{defn}
\label{stonglybistabil}
A monomial ideal $I\subset R=k[s,t,u,v]$ is strongly bistable if for
every monomial $m\in I$ the following conditions are satisfied:
\begin{enumerate}
\item if $m$ is divisible by $t$, then $sm/t\in I$.
\item if $m$ is divisible by $v$, then $um/v\in I$ .
\end{enumerate}
\end{defn}

As in the $\mathbb{Z}$-graded case, the ideal $bigin(I)$ has the same 
bigraded Hilbert function as $I$. Propositions 1.5 and 1.6 
of \cite{ACD} show that $bigin(I)$ is bi-Borel fixed, and in 
characteristic zero, $bigin(I)$ is strongly bistable.

\begin{prop}\label{bigin}
For each of the Hilbert functions in Proposition \ref{HFnoLin} there are exactly two strongly bistable monomial ideals realizing it. These ideals and their respective bigraded resolutions are:
\begin{enumerate}
\item  $G_1=\langle s^2u,s^2v,stu,stv,t^2u^2,t^2uv,t^3u,t^3v,t^2v^3\rangle $ with minimal resolution
\begin{small}
 \begin{equation}\label{bigin1}
 0 \leftarrow G_1 \leftarrow  \begin{array}{c}
(-2,-1)^4 \\
 \oplus \\  
 (-2,-2)^2\\
 \oplus \\  
 (-2,-3)\\
 \oplus \\  
 (-3,-1)^2\\
 \end{array} \longleftarrow 
  \begin{array}{c}
 (-2,-2)^2\\
 \oplus \\  
 (-2,-3)\\
  \oplus \\ 
 (-2,-4)\\
 \oplus \\
 (-3,-1)^2\\
 \oplus \\
(-3,-2)^5\\
 \oplus \\
 (-3,-3)^2\\
 \oplus \\
(-4,-1)^2\\
\end{array} \longleftarrow 
\begin{array}{c}
 (-3,-2)\\
 \oplus \\
 (-3,-3)^2\\
  \oplus \\
(-3,-4)^2\\
 \oplus \\
(-4,-2)^3\\
 \oplus \\
(-4,-3)\\
 \end{array} \longleftarrow 
 \begin{array}{c}
 (-4,-3)\\
  \oplus \\
(-4,-4) \end{array}\leftarrow 0
\end{equation}
\end{small}
 $G'_1=\langle s^2u,s^2v,stu,t^2u,stv^2,st^2v,t^3v,t^2v^3\rangle $ with minimal resolution
\begin{small}
 \begin{equation}\label{bigin1'}
 0 \leftarrow G'_1 \leftarrow  \begin{array}{c}
(-2,-1)^4 \\
 \oplus \\  
 (-2,-2)\\
 \oplus \\  
 (-2,-3)\\
 \oplus \\  
 (-3,-1)^2\\
 \end{array} \longleftarrow 
  \begin{array}{c}
 (-2,-2)\\
 \oplus \\  
 (-2,-3)\\
  \oplus \\ 
 (-2,-4)\\
 \oplus \\
 (-3,-1)^2\\
 \oplus \\
(-3,-2)^4\\
 \oplus \\
 (-3,-3)^2\\
 \oplus \\
(-4,-1)^2\\
\end{array} \longleftarrow 
\begin{array}{c}
 (-3,-3)^2\\
  \oplus \\
(-3,-4)^2\\
 \oplus \\
(-4,-2)^3\\
 \oplus \\
(-4,-3)\\
 \end{array} \longleftarrow 
 \begin{array}{c}
 (-4,-3)\\
  \oplus \\
(-4,-4) \end{array}\leftarrow 0
\end{equation}
\end{small}
 
\item  $G_2=\langle s^2u,s^2v,stu,stv,t^2u^2,t^2uv,t^3u,t^3v\rangle $ with minimal resolution
\begin{small}
 \begin{equation}\label{bigin2}
 0 \leftarrow G_2 \leftarrow  \begin{array}{c}
(-2,-1)^4 \\
 \oplus \\  
 (-2,-2)^2\\
 \oplus \\    
 (-3,-1)^2\\
 \end{array} \longleftarrow 
  \begin{array}{c}
 (-2,-2)^2\\
 \oplus \\  
 (-2,-3)\\
  \oplus \\ 
 (-3,-1)^2\\
 \oplus \\
(-3,-2)^5\\
 \oplus \\
(-4,-1)^2\\
\end{array} \longleftarrow 
\begin{array}{c}
 (-3,-2)\\
 \oplus \\
 (-3,-3)^2\\
  \oplus \\
(-4,-2)^3\\
 \end{array} \longleftarrow 
 \begin{array}{c}
 (-4,-3)\\
 \end{array}\leftarrow 0
\end{equation}
\end{small}
$G'_2=\langle s^2u,s^2v,stu,t^2u,stv^2,st^2v,t^3v \rangle$ with minimal resolution
\begin{small}
\begin{equation}\label{bigin2'}
0 \leftarrow G'_2 \leftarrow  \begin{array}{c}
(-2,-1)^4 \\
 \oplus \\  
 (-2,-2)\\
 \oplus \\  
 (-3,-1)^2\\
 \end{array} \longleftarrow 
  \begin{array}{c}
 (-2,-2)\\
 \oplus \\
 (-2,-3)\\
 \oplus \\
 (-3,-1)^2\\
 \oplus \\
(-3,-2)^4\\
 \oplus \\
(-4,-1)^2\\
\end{array} \longleftarrow 
\begin{array}{c}
 (-3,-3)^2\\
 \oplus \\
(-4,-2)^3\\
 \end{array} \longleftarrow 
(-4,-3) \leftarrow 0 
\end{equation}
\end{small}
\end{enumerate}
\end{prop}

 \begin{proof}
There are only two strongly bistable sets of four monomials in $R_{2,1}$: $\{s^2u, s^2v,$ $stu, stv\}$ and $\{s^2, s^2v, stu, t^2u\}$. To complete $\{s^2u, s^2v,$ $stu, stv\}$ to an ideal realizing one of the Hilbert functions in Proposition \ref{HFnoLin} we need two additional monomials in $R_{2,2}$, which must be $t^2u^2, t^2uv$ in order to preserve bistability. 
Then we must add the two remaining monomials $t^3u,t^3v$ in $R_{3,1}$,
which yields the second Hilbert function. To realize the first 
Hilbert function we must also include the 
remaining monomial $t^2v^3 \in R_{2,3}$. 
To complete $\{s^2, s^2v, stu, t^2u\}$ to an ideal realizing one of the Hilbert functions in Proposition \ref{HFnoLin}, we need one additional monomial in $R_{2,2}$ which must be $stv^2$ in order to preserve bistability. Then 
we must add the two remaining monomials $st^2v,t^3v \in R_{3,1}$. Then to 
realize the first Hilbert function, we must add the remaining monomial 
$t^2v^3 \in R_{2,3}$.
\end{proof}

\begin{thm}\label{T1T2res}
There are two numerical types for the minimal Betti numbers of basepoint free ideals $I_U$ without linear syzygies.
\begin{enumerate}
\item If there is a bidegree $(0,2)$  first syzygy  then $I_U$ has numerical Type 2.
\item If there is no bidegree $(0,2)$ first syzygy then $I_U$ has numerical Type 1.
\end{enumerate}
\end{thm}

\begin{proof}
Proposition \ref{HFnoLin} establishes that the two situations above are the only possibilities in the absence of linear syzygies and gives the Hilbert function corresponding to each of the two cases. Proposition \ref{bigin}  identifies the possible bigeneric initial ideals for each case. Since these bigeneric initial ideals are initial ideals obtained following a change of coordinates, the cancellation principle applies. We now show the resolutions  (\ref{bigin1}), (\ref{bigin1'}) must cancel to the Type 1 resolution and  the resolutions  (\ref{bigin2}), (\ref{bigin2'}) must cancel to the Type 2 resolution. 

Since $I_U$ is assumed to have no linear syzygies, all linear syzygies appearing in the resolution of its bigeneric initial ideal must cancel. 
Combined with Proposition \ref{syznoLin}, this establishes that in (\ref{bigin2}) or (\ref{bigin2'}) the linear cancellations are the only ones that occur.  In (\ref{bigin1}), the cancellations of generators and first syzygies in bidegrees  $(2,2), (2,3), (3,1)$ are obvious. The  second syzygy in bidegree $(3,2)$ 
depends on the cancelled first syzygies, therefore it must also be cancelled. 
This is natural, since by Proposition \ref{syznoLin}, there are exactly 
four bidegree $(3,2)$ first syzygies. An examination of the maps in the 
resolution (\ref{bigin1}) shows that the bidegree $(3,3)$ second syzygies 
depend on the cancelled first syzygies, so they too must cancel. Finally the bidegree $(4,3)$ last syzygy depends on the previous cancelled second syzygies and so must also cancel.

In (\ref{bigin1'}), the cancellations of generators and first syzygies in bidegrees  $(2,2)$, $(2,3)$, $(3,1)$ are obvious. The second syzygies of bidegree $(3,3)$ depend only on the 
cancelled first syzygies, so they too cancel. Finally the bidegree $(4,3)$ last syzygy depends on the previous cancelled second syzygies and so it must also cancel.
\end{proof}

\section{Primary decomposition}

\begin{lem}\label{PD1}
If $U$ is basepoint free, all embedded primes of $I_U$ are of the
form 
\begin{enumerate}
\item $\langle s,t,l(u,v) \rangle$
\item $\langle u,v,l(s,t) \rangle$
\item $\mathfrak{m} = \langle s,t,u,v \rangle$
\end{enumerate}
\end{lem}
\begin{proof}
Since $\sqrt{I_U} = \langle s,t\rangle \cap \langle u,v \rangle$,
an embedded prime must contain either $\langle s,t\rangle$
or $\langle u,v\rangle$ and  modulo these ideals any remaining minimal generators can be considered as irreducible polynomials in
 $k[u,v]$ or $k[s,t]$ (respectively). But  the only prime ideals here
are $\langle l_i \rangle$ with $l_i$ a linear form, or the irrelevant
ideal. 
\end{proof}

\begin{lem}\label{PDD1}
If $U$ is basepoint free, then the primary components corresponding to 
minimal associated primes of $I_U$ are 
\begin{enumerate}
\item $Q_1=\langle u,v \rangle$
\item $Q_2=\langle s,t \rangle^2$ or $Q_2=\langle p,q \rangle$, with $p,q \in R_{2,0}$ \mbox{ and }$\sqrt{p,q} = \langle s,t \rangle$.
\end{enumerate}
\end{lem}
\begin{proof}

Let $Q_1, Q_2$ be the primary components associated to $\langle u,v \rangle$ and $\langle s,t \rangle$ respectively. Since $I_U\subset Q_1\subset \langle u,v \rangle^m$  and $I_U$ is generated in bidegree $(2,1)$,  $Q_1$ must contain at least one element of bidegree $(0,1)$. If $Q_1$ contains exactly one element $p(u,v)$ of bidegree  $(0,1)$, then $V$ is contained in the fiber of  $\Sigma_{2,1}$ over the point $V(p(u,v))$, which contradicts the basepoint free assumption. Therefore $Q_1$ must contain two independent linear forms in $u,v$ and hence $Q_1=\langle u,v \rangle$. 

Since $I_U\subset Q_2$ and $I_U$ contains elements of bidegree $(2,1)$,  $Q_2$ must contain at least one element of bidegree $(2,0)$. If $Q_2$ contains exactly one element $q(s,t)$ of bidegree  $(2,0)$, then $V$ is contained in the fiber of  $\Sigma_{2,1}$ over the point $V(q(s,t))$, which contradicts the basepoint free assumption. If $Q_2$ contains exactly two elements of bidegree  $(2,0)$ which share a common linear factor $l(s,t)$, then $I_U$ is contained in the ideal $\left<l(s,t)\right>$, which contradicts the basepoint free assumption as well.  Since the bidegree $(2,0)$ part of $Q_2$ is contained in the linear span of $s^2,t^2,st$, it follows that the only possibilities consistent with the conditions above are  $Q_2=\langle p,q \rangle$ with $\sqrt{p,q} =\langle s,t \rangle$ or $Q_2=\langle s^2,t^2,st \rangle$.
\end{proof}

\begin{prop}\label{PD2}
For each type of minimal free resolution of $I_U$ with $U$ basepoint
free, the embedded primes of $I_U$ are as in Table 1.
\end{prop}
\begin{proof}
First observe that $\mathfrak{m}=\langle s,t,u,v \rangle$ 
is an embedded prime for
each of Type 1 to Type 4. This follows since the respective
free resolutions have length four, so 
\[
Ext^4_R(R/I_U,R) \ne 0.
\]
By local duality, this is true iff $H^0_{\mathfrak{m}}(R/I_U) \ne 0$ 
iff $\mathfrak{m} \in \Ass(I_U)$. Since the resolutions for Type 5
and Type 6 have projective dimension less than four, 
this also shows that in Type 5 and Type 6, $\mathfrak{m} \not\in \Ass(I_U)$.
Corollary~\ref{T5PD} and Proposition~\ref{LS2} show the embedded 
primes for Type 5 and Type 6 are as in Table 1.

Thus, by Lemma~\ref{PD1}, all that remains is to study primes of
the form $\langle s,t,L(u,v) \rangle$ and $\langle u,v,L(s,t) \rangle$
for Type 1 through 4. For this, suppose 
\[
I_U = I_1 \cap I_2 \cap I_3, \mbox{ where}
\]
\begin{enumerate}
\item $I_1$ is the intersection of primary components corresponding 
to the two minimal associated primes identified in Lemma~\ref{PDD1}. 
\item $I_2$ is the intersection of embedded primary components 
not primary to $\mathfrak{m}$.
\item $I_3$ is primary to $\mathfrak{m}$.
\end{enumerate}
By Lemma~\ref{PDD1},
if $I_1 = \langle u,v \rangle \cap \langle p,q \rangle$ with 
$\sqrt{p,q} = \langle s,t \rangle$, then $I_1$ is basepoint free
and consists of four elements of bidegree $(2,1)$, thus $I_U = I_1$ and
has Type 6 primary decomposition. So we may assume $I_1 =\langle u,v \rangle \cap \langle s,t\rangle^2$.
Now we switch gears and consider all ideals in the $\mathbb{Z}$--grading
where the variables have degree one. In the  $\mathbb{Z}$--grading
\[
HP(R/I_1,t) = 4t+2.
\]
Since the Hilbert polynomials of $R/(I_1\cap I_2)$ and 
$R/I_U$ are identical, we can compute the Hilbert polynomials of
$R/(I_1\cap I_2)$ for Type 1 through 4 using Theorems~\ref{T1T2res}, 
\ref{Type4MC} and  \ref{T3res}. For example, in Type 1, the
bigraded minimal free resolution is
\[
0 \leftarrow I_U \leftarrow (-2,-1)^4 \longleftarrow \begin{array}{c}
 (-2,-4)\\
 \oplus \\
(-3,-2)^4\\
 \oplus \\
(-4,-1)^2\\
\end{array} \longleftarrow 
\begin{array}{c}
 (-3,-4)^2\\
 \oplus \\
(-4,-2)^3\\
 \end{array} \longleftarrow 
(-4,-4) \leftarrow 0. 
\]
Therefore, the  $\mathbb{Z}$--graded minimal free resolution is 
\[
0 \leftarrow I_U \leftarrow (-3)^4 \longleftarrow \begin{array}{c}
 (-6)\\
 \oplus \\
(-5)^6\\
 \end{array} \longleftarrow 
\begin{array}{c}
 (-7)^2\\
 \oplus \\
(-6)^3\\
 \end{array} \longleftarrow 
(-8) \leftarrow 0. 
\]
Carrying this out for the other types shows that
the  $\mathbb{Z}$--graded Hilbert polynomial of $R/(I_1 \cap I_2)$ is  
\begin{enumerate}
\item In Type 1 and Type 3:
\[
HP(R/(I_1 \cap I_2),t) = 4t+2.
\]
\item In Type 2 and Type 4:
\[
HP(R/(I_1 \cap I_2),t) = 4t+3.
\]
\end{enumerate}
In particular, for Type 1 and Type 3,  
\[
HP(R/I_1,t) = HP(R/(I_1 \cap I_2),t),
\] 
and in Type 2 and Type 4, the Hilbert polynomials differ by one:
\[
HP(I_1/(I_1 \cap I_2), t) = 1.
\]
Now consider the short exact sequence
\[
0 \longrightarrow I_1 \cap I_2 \longrightarrow I_1 \longrightarrow 
I_1/(I_1 \cap I_2) \longrightarrow 0.
\]
Since $I_1\cap I_2 \subseteq I_1$, in Type 1 and Type 3 where
the Hilbert Polynomials are equal, there can be no embedded
primes save $\mathfrak{m}$. In Type 2 and Type 4, 
since $HP(I_1/(I_1 \cap I_2),t)=1$, $I_1/(I_1 \cap I_2)$ is supported 
at a point of $\P^3$ which corresponds to a codimension three 
prime ideal of the form $\langle l_1,l_2,l_3\rangle$. Switching 
back to the fine grading, by Lemma~\ref{PD1}, this 
prime must be either $\langle s,t,l(u,v)\rangle$
or $\langle u,v,l(s,t) \rangle$. Considering the multidegrees 
in which the Hilbert function of $I_1/(I_1 \cap I_2)$ is nonzero
shows that the embedded prime is of type $\langle s,t,l(u,v)\rangle$.
\end{proof}

\section{The Approximation complex and Implicit equation of $X_U$}
The method of using moving lines and moving quadrics to 
obtain the implicit equation of a curve or surface was
developed by Sederberg and collaborators in \cite{sc},
\cite{sgd}, \cite{ssqk}. In \cite{cox2}, Cox gives a nice
overview of this method and makes explicit the connection
to syzygies. In the case of tensor product surfaces these
methods were first applied by Cox-Goldman-Zhang in \cite{cgz}. 
The approximation complex was introduced by Herzog-Simis-Vasconcelos in
\cite{hsv1},\cite{hsv2}.  From a mathematical perspective, 
the relation between the implicit equation and syzygies comes 
from work of Bus\'e-Jouanolou \cite{bj} and Bus\'e-Chardin \cite{bc} 
on approximation complexes and the Rees algebra; their work was
extended to the multigraded setting in  \cite{bot}, \cite{bdd}.
The next theorem follows from work of Botbol-Dickenstein-Dohm \cite{bdd} 
on toric surface parameterizations, and also from a more general
result of Botbol \cite{bot}. The novelty of our approach is 
that by obtaining an explicit description of the syzygies, we
obtain both the implicit equation for the surface and a description
of the singular locus. Theorem~\ref{SingX} gives a particularly
interesting connection between syzygies of $I_u$ and singularities of $X_U$. 
\begin{thm}\label{ImX}
If $U$ is basepoint free, then the implicit equation 
for $X_U$ is determinantal, obtained from the $4 \times 4$ minor 
of the first map of the approximation complex $\mathcal{Z}$
in bidegree $(1,1)$, except for Type 6, where $\phi_U$ is not birational.
\end{thm}
\subsection{Background on approximation complexes}
We give a brief overview of approximation complexes,
for an extended survey see \cite{c}. For
\[
I = \langle f_1, \ldots, f_n \rangle \subseteq R=k[x_1,\ldots x_m],
\]
let $K_i \subseteq \Lambda^i(R^n)$ be the kernel of 
the $i^{th}$ Koszul differential on $\{f_1, \ldots, f_n\}$, 
and $S = R[y_1,\ldots, y_n]$. Then
the approximation complex $\mathcal{Z}$ has $i^{th}$ term 
\[
\mathcal{Z}_i = S \otimes_R K_i.
\]
The differential is the Koszul differential on $\{y_1, \ldots, y_n\}$.
It turns out that $H_0(\mathcal{Z})$ is $S_I$ and  the higher homology
depends (up to isomorphism) only on $I$. For $\mu$ a bidegree in $R$, define
\[
\mathcal{Z}^{\mu} \mbox{ : } \cdots \longrightarrow 
k[y_1,\ldots, y_n]\otimes_k (K_i)_\mu \stackrel{d_i}{\longrightarrow}k[y_1,\ldots, y_n]\otimes_k (K_{i-1})_\mu  \stackrel{d_{i-1}}{\longrightarrow}\cdots
\]
If the bidegree $\mu$ and base locus of $I$ satisfy certain conditions,
then the determinant of $\mathcal{Z}^{\mu}$ is a power of the implicit
equation of the image. This was first proved in \cite{bj}. In Corollary 14 
of \cite{bdd}, Botbol-Dickenstein-Dohm give a specific bound for $\mu$ in
the case of a toric surface and map with zero-dimensional base locus and
show that in this case the gcd of the maximal minors of $d_1^{\mu}$ is
the determinant of the complex. For four sections of bidegree $(2,1)$,
the bound in \cite{bot} shows that $\mu =(1,1)$. To make things concrete, we 
work this out for Example~\ref{ex1}.
\begin{exm}\label{exLast}
Our running example is $U = \spn \{s^2u, s^2v, t^2u, t^2v+stv \}$.
Since $K_1$ is the module of syzygies on $I_U$, which is
generated by the columns of 
\[
\left[\begin{matrix}
-v &-t^2 &0&     0 &-tv \\
u & 0 & -st-t^2 &0 & 0 \\
 0 & s^2 & 0 &-sv-tv & -sv\\
0 & 0 & s^2 &tu & su 
\end{matrix}
 \right]
\]
The first column encodes the relation $ux_1-vx_0 =0$, then next four 
columns the relations
\[
\begin{array}{ccc}
s^2x_2-t^2x_0 &=& 0 \\
s^2x_3 -(st+t^2)x_1 &=&0 \\
tu x_3 -(sv+tv)x_2 &=&0\\
su x_3 -svx_2-tvx_0 &=& 0
\end{array}
\]
If we were in the singly graded case, we would need to use $\mu = 2$, 
and a basis for $\mathcal{Z}_1^2$ consists of $\{s,t,u,v\} \cdot ux_1-vx_0$,
and the remaining four relations. With respect to the ordered basis 
$\{s^2,st,t^2, su,sv,tu,tv, u^2,uv,v^2\}$ for $R_2$ and writing
$\cdot$ for $0$, the matrix for 
$d_1^{2} : \mathcal{Z}_1^2 \longrightarrow \mathcal{Z}_0^2$ is
\vskip .02in
\[
\left[\begin{matrix}
\cdot & \cdot & \cdot & \cdot & x_2   & x_3 & \cdot& \cdot\\
\cdot & \cdot & \cdot & \cdot & \cdot &-x_1 & \cdot& \cdot\\
\cdot & \cdot & \cdot & \cdot & -x_0 & -x_1 & \cdot& \cdot\\
x_1 & \cdot & \cdot & \cdot & \cdot & \cdot & \cdot& x_3\\
-x_0 & \cdot & \cdot & \cdot & \cdot & \cdot & -x_2& -x_2\\
\cdot & x_1 & \cdot & \cdot & \cdot & \cdot &  x_3 & \cdot\\
\cdot & -x_0 & \cdot & \cdot & \cdot & \cdot & -x_2 & -x_0\\
\cdot &  \cdot & x_1 & \cdot & \cdot & \cdot & \cdot& \cdot \\
\cdot &  \cdot &-x_0 & x_1 & \cdot & \cdot & \cdot & \cdot\\
\cdot &  \cdot &\cdot &-x_0  & \cdot & \cdot & \cdot & \cdot\\
\end{matrix}
 \right]
\]
\vskip .05in
However, this matrix represents all the first syzygies of total
degree two. Restricting to the submatrix of bidegree $(1,1)$ syzygies
corresponds to choosing rows indexed by $\{su,sv,tu,tv\}$, yielding
\vskip .02in
\[
\left[\begin{matrix}
x_1 & \cdot & \cdot& x_3\\
-x_0 & \cdot & -x_2& -x_2\\
\cdot & x_1 &  x_3 & \cdot\\
\cdot & -x_0 & -x_2 & -x_0\\
\end{matrix}
 \right]
\]
\end{exm}
We now study the observation made in the introduction,
that linear syzygies manifest in a linear singular locus. 
Example~\ref{exLast} again provides the key intuition: a
linear first syzygy gives rise to two columns of $d^{(1,1)}_1$.
\begin{thm}\label{SingX}
If $U$ is basepoint free and $I_U$ has a unique
linear syzygy, then the codimension one singular locus of $X_U$ 
is a union of lines.
\end{thm}
\begin{proof}
Without loss of generality we assume the linear syzygy involves the first two 
generators $p_0,p_1$ of $I_U$, so that in the two remaining 
columns corresponding to the linear syzygy the only nonzero 
entries are $x_0$ and $x_1$, which appear exactly as in 
Example~\ref{exLast}. Thus, in bidegree $(1,1)$, the matrix for $d^{(1,1)}_1$ 
has the form
\[
\left[\begin{matrix}
x_1 & \cdot & \ast & \ast\\
-x_0 & \cdot & \ast & \ast\\
\cdot &x_1 &  \ast & \ast\\
\cdot &-x_0 &  \ast & \ast\\
\end{matrix}
 \right]
\]
Computing the determinant using two by two minors in the two left
most columns shows the implicit equation of $F$ is of the form
\[
x_1^2 \cdot f + x_0x_1 g + x_0^2 h,
\]
which is singular along the line $\V(x_0,x_1)$. To show that
the entire singular locus is a union of lines when $I_U$ has 
resolution Type $\in \{3,4,5\}$, we must analyze the structure
of $d_1^{(1,1)}$. For Type 3 and 4, Theorems~\ref{T3res} and~\ref{Type4MC}
give the first syzygies, and show that the implicit equation 
for $X_U$ is given by the determinant of 
\[
\left[\begin{matrix}
 -x_1 & \cdot & x_2 & x_3  \\
 \cdot & -x_1 & a_1x_2-b_1x_0 & a_1x_3-c_1x_0  \\
 x_0 & \cdot & a_2x_2-b_0x_1&a_2x_3-c_0x_1\\
 \cdot & x_0 & a_3x_2-b_2x_0-b_3x_1 & a_3x_3-c_2x_0-c_3x_1\\
\end{matrix}
 \right].
\]
We showed above that $\V(x_0,x_1) \subseteq \Sing(X_U)$. Since 
$X_U \setminus \V(x_0,x_1) \subseteq U_{x_0} \cup U_{x_1}$, it
suffices to check that $X_U \cap U_{x_0}$ and $X_U \cap U_{x_1}$ 
are smooth in codimension one. $X_U \cap U_{x_0}$ is defined by
\[
\begin{array}{ccc}
-{c}_{1} {y}_{2}+{b}_{1} {y}_{3}& = &(-{b}_{3} {c}_{0}+{b}_{0} {c}_{3}) {y}_{1}^{4}+({a}_{3} {c}_{0}-{a}_{2} {c}_{3}) {y}_{1}^{3} {y}_{2} \\
                             & + & (-{a}_{3} {b}_{0}+{a}_{2} {b}_{3}) {y}_{1}^{3}{y}_{3}  +  +(-{b}_{2} {c}_{0}+{b}_{0} {c}_{2}){y}_{1}^{3}\\
                             &+ & ({a}_{1} {c}_{0}-{a}_{2} {c}_{2}-{c}_{3}) {y}_{1}^{2} {y}_{2}+(-{a}_{1} {b}_{0}+{a}_{2}{b}_{2}+{b}_{3}) {y}_{1}^{2} {y}_{3}\\
 & + &(-{b}_{1} {c}_{0}+{b}_{0} {c}_{1}) {y}_{1}^{2}+(-{a}_{2} {c}_{1}-{c}_{2}) {y}_{1} {y}_{2}+({a}_{2} {b}_{1}+{b}_{2}) {y}_{1} {y}_{3}.
\end{array}
\]
By basepoint freeness, $b_1$ or $c_1$ is nonzero, as is $(-{b}_{3} {c}_{0}+{b}_{0} {c}_{3})$, 
so in fact $X_U \cap U_{x_0}$ is smooth. A similar calculation shows that $X_U \cap U_{x_1}$
is also smooth, so for Type 3 and Type 4, $\Sing(X_U)$ is a line. 
In Type 5 the computation is more cumbersome: with notation as in 
Proposition~\ref{LS3}, the relevant $4 \times 4$ submatrix is 
\[
\left[\begin{matrix}
 x_1 & \cdot & a_1x_3-a_3x_2 & \beta_1\alpha_2x_1 + \alpha_1\alpha_2x_0  \\
 -x_0 & \cdot & b_1x_3-b_3x_2 & \beta_1\beta_2x_1 + \alpha_1\beta_2x_0  \\
 \cdot & x_1 & \beta_1\alpha_2x_1 + \alpha_1\alpha_2x_0 &a_2x_3-a_4x_2\\
 \cdot & -x_0 & \beta_1\beta_2x_1 + \alpha_1\beta_2x_0 &b_2x_3-b_4x_2\\
\end{matrix}
 \right].
\]
A tedious but straightforward calculation shows that $\Sing(X_U)$ 
consists of three lines in Type 5a and  a pair of lines in Type 5b.
\end{proof}
\begin{thm}\label{SingX2}
If $U$ is basepoint free, then the codimension one singular locus of $X_U$ 
is as described in Table 1.
\end{thm}
\begin{proof}
For a resolution of Type 3,4,5, the result follows from Theorem~\ref{SingX},
and for Type 6 from Proposition~\ref{LS2}. For the generic case (Type 1),
the result is obtained by Elkadi-Galligo-Le \cite{egl}, so
it remains to analyze Type 2. By Lemma~\ref{02syz}, the $(0,2)$
first syzygy implies that we can write $p_0, p_1, p_2$ as $\alpha
 u, \beta v, \alpha v + \beta u$ for some $\alpha, \beta \in R_{2,0}$.
Factor $\alpha$ as product of two linear forms in $s$ and $t$, so 
after a linear change of variables we may assume $\alpha = s^2$ or $st$.

If $\alpha = s^2$, write $\beta=(ms+nt)(m's+n't)$, and note that $n$ and $n'$ 
cannot both vanish, because then $\beta$ is a scalar multiple of $\alpha$, 
violating linear independence of the $p_i$. Thus, after a linear change 
of variables, we may assume $\beta$ is of the form $t(ks+lt)$, so the 
$p_i$'s are of the form
\[
\{s^2u, (ks+lt)tv, s^2v+(ks+lt)tu, p_3\}
\]
If $l=0$, then $I_U = \langle s^2u, kstv, s^2v+kstu, p_3 \rangle$, 
which is not basepoint free: if $s=0$, thus the first 3 polynomials vanish 
and $p_3$ becomes $t^2(au+bv)$ which vanishes for some  
$(u:v) \in \mathbb{P}^1$. So $l \ne 0$, and after a linear change of 
variables $t \mapsto \frac{t}{l}$ and  $u \mapsto lu$, we may assume $l=1$ and hence
\[
I_U = \langle s^2u, t^2v+kstv, s^2v+t^2u+kstu, p_3\rangle.
\]
Now consider two cases, $k=0$ and $k\not = 0$. In the latter, we may 
assume $k=1$: first replace $ks$ by $s$ and then replace $k^2u$ by $u$. 
In either case,  reducing $p_3$ by the other generators of $I_U$ shows
we can assume
\[
p_3=astu+bs^2v+cstv = s(atu+bsv+ctv).
\]
By Theorem~\ref{T1T2res} there is one first syzygy of bidegree $(0,2)$, two
first syzygies of bidegree $(2,0)$, and four first syzygies of bidegree 
$(1,1)$. A direct calculation shows that two of the bidegree $(1,1)$ 
syzygies are 
$$ \left[ \begin{array}{cc}atu+bsv+ctv &0 \\  0  & asu-btu+csv \\0 & btv  \\ -su
 & -tv  \end{array} \right]. $$ 
The remaining syzygies depend on $k$ and $a$. For example, if $k=a=0$, then the
remaining bidegree $(1,1)$ syzygies are 
$$ \left[ \begin{array}{cc}0 & -b^3tv \\  b^2su+c^2sv & c^3sv \\-b^2sv  & -b^2cs
v \\ bsv-ctv & b^2tu+bcsv-c^2tv   \end{array} \right], $$
and the bidegree $(2,0)$ syzygies are 
$$ \left[ \begin{array}{cc}0 & b^3t^2\\     bs^2+cst  & -c^3st \\0 & -b^3s^2  \\
  -t^2    & b^2s^2-bcst+c^2t^2    \end{array} \right] $$
Thus, if $k=a=0$, using the basis $\{ su, tu, sv, tv \}$ for $R_{1,1}$, the
matrix whose determinant gives the implicit equation for $X_U$ is 
\[
\left[
\begin{matrix}
-x_3 & 0 &  b^2x_1 & 0 \\
  0 & -bx_1 & 0 & b^2x_3\\
 bx_0 & cx_1 & c^2x_1-b^2x_2+bx_3& c^3x_1-b^2cx_2+bcx_3\\
 cx_0 & bx_2-x_3 & -cx_3 & -b^3x_0-c^2x_3 \\
\end{matrix}
\right]
\]
Since $k=a=0$, if both $b$ and $c$ vanish there will be a linear 
syzygy on $I_U$, contradicting our assumption. So suppose $b\ne 0$ and
scale the generator $p_3$ so $b=1$:
\[
\left[
\begin{matrix}
-x_3 & 0 &  x_1 & 0 \\
  0 & -x_1 & 0 & x_3\\
 x_0 & cx_1 & c^2x_1-x_2+x_3& c^3x_1-cx_2+cx_3\\
 cx_0 & x_2-x_3 & -cx_3 & -x_0-c^2x_3 \\
\end{matrix}
\right]
\]
Expanding along the top two rows by $2 \times 2$ minors as in the
proof of Theorem~\ref{SingX} shows that $X_U$ is singular along
${\bf V}(x_1,x_3)$, and evaluating the Jacobian matrix with $x_1=0$
shows this is the only component of the codimension one singular
locus with $x_1 = 0$. Next we consider the affine patch $U_{x_1}$. 
On this patch, the Jacobian ideal is 
\[
\langle (4 {x}_{3}+c^{2}) ({x}_{2}-2 {x}_{3}-c^{2}),({x}_{2}-2 {x}_{3}-c^{2}) ({x}_{2}+2 {x}_{3}),2 {x}_{0}-2 {x}_{2}{x}_{3}-{x}_{2} c^{2}+2 {x}_{3}^{2}+4 {x}_{3} c^{2}+c^{4}
\rangle
\]
which has codimension one component given by 
\[{\bf V}({x}_{2}-2 {x}_{3}-c^{2}, x_0-x_3^2),
\]
a plane conic. Similar calculations work for the other cases.
\end{proof}
\section{Connection to the dual scroll}
We close by connecting our work to the results of Galligo-L\^e
in \cite{gl}. First, recall that the ideal of $\Sigma_{2,1}$ is 
defined by the two by two minors of 
\[
\left[\begin{matrix} x_0 & x_1 & x_2 \\ x_3 & x_4 &x_5 
\end{matrix}\right].
\]
Combining this with the relations $x_1^2-x_0x_2$ and $x_4^2-x_3x_5$
arising from $\nu_2$ shows that the image of the map $\tau$ defined in 
Equation~\ref{e1} is the vanishing locus of the two by two minors of 
\begin{equation}\label{SV}
\left[\begin{matrix} x_0 & x_1 & x_3 & x_4 \\ x_1 & x_2 & x_4 &x_5 
\end{matrix}\right].
\end{equation}

Let $A$ denote the $4 \times 6$ matrix of 
coefficients of the polynomials defining $U$ in the monomial
basis above. We regard $\P(U) \hookrightarrow \P(V)$ via ${\bf a} \mapsto {\bf a} \cdot A$. 
Note that $I_U$ and the implicit equation of $X_U$  are independent of the choice of generators $p_i$ (\cite{c}).

The dual projective space of $\P(V)$ is $\P(V^*)$ where $V^*=Hom_k(V,k)$ and the projective subspace of $\P(V^*)$ orthogonal to $U$ is defined to be $\P((V/U)^*)=\P(U^\perp)$, where $U^\perp=\{f\in V^* |f(u)=0, \forall u\in U\}$ is algebraically described as the kernel of $A$. The elements
of $U^\perp$ define the space of linear forms in $x_i$
which vanish on $\P(U)$. In Example~\ref{ex1}  
$U = \spn \{s^2u,s^2v,t^2u,t^2v+stv \}$, so $A$ is the matrix
\[
\left[\begin{matrix} 
1 & 0 & 0 & 0 & 0 & 0 \\ 
0 & 0 & 0 & 1 & 0 & 0 \\ 
0 & 0 & 1 & 0 & 0 & 0 \\ 
0 & 0 & 0 & 0 & 1 & 1 \\
\end{matrix}\right],
\]
and $\P(U) = \V(x_1,x_4-x_5) \subseteq \P^5$.

The conormal variety $N(X)$ is the incidence variety defined as the closure of the set of pairs $\{(x,\pi)\in \P(V)\times \P(V^*)\}$ such that $x$ is a smooth point of $X$ and $\pi$ is an element of the linear subspace orthogonal to the tangent space $T_{X,x}$ in the sense described above. $N(X)$ is irreducible and  for a varieties $X$ embedded in $\P(V)\simeq \P^5$ the  dimension of $N(X)$ is 4. The image of the projection of $N(X)$ onto the factor $\P(V^*)$ is by definition the dual variety of $X$ denoted $X^*$.

Denote by $\widetilde{X_U}$ the variety $X_U$ re-embedded as a hypersurface of $\P(U)$.  Proposition 1.4 (ii) of \cite{CRS} applied to our situation reveals

\begin{prop}
If $\widetilde{X_U}^*\subset {\P^3}^*$ is a hypersurface which is swept by a one dimensional family of lines, then $\widetilde{X_U}^*$ is either a 2-dimensional scroll or else a curve.
\end{prop}

The cited reference includes precise conditions that allow the two possibilities to be distinguished.  Proposition 1.2 of \cite{CRS} reveals the relation between $X_U^*\subset \P(V^*)$ and $\widetilde{X_U}^*\subset \P(U^*)$, namely
\begin{prop}
In the above setup, $X_U^*\subset \P(V^*)$ is a cone over $\widetilde{X_U}^*\subset \P(U^*)$ with vertex $U^\perp=\P((V/U)^*)$.
\end{prop}

It will be useful for this reason to consider the map  $\pi:\P(V)\rightarrow \P(U^*)$ defined by $\pi(p)=(\ell_1(p):\ldots:\ell_4(p))$, where $\ell_1,\ldots, \ell_4$ are the defining equations of $U^\perp$. The map $\pi$ is projection
from $\P(U^\perp)$ and $\pi(X_U^*) = \widetilde{X_U}^*$. Using a direct approach  Galligo-L\^e obtain in\cite{gl} that $\pi^{-1}(\widetilde{X_U}^*)$ is a (2,2)-scroll in $\P(V^*)$ which they denote by $\FF_{2,2}^*$.  For brevity we write $\FF$ for $\pi^{-1}(\widetilde{X_U}^*)$.

 Galligo-L\^e classify possibilities for the implicit equation 
of $X_U$ by considering the pullback $\phi_U^*$ of 
$\P(U^*) \cap \FF$ to $(\PP)^*$. The two linear forms $L_i$ 
defining $\P(U^*)$ pull back to give a pair of 
bidegree $(2,1)$ forms on $(\PP)^*$. 
\begin{prop}[\cite{gl} \S 6.5]\label{GL} If $\phi_U^*(L_1) \cap \phi_U^*(L_2)$ is infinite, 
then $\phi_U^*(L_1)$ and $\phi_U^*(L_2)$ share a common factor $g$, for which 
the possibilities are:
\begin{enumerate}
\item $deg(g) = (0,1)$.
\item $deg(g) = (1,1)$ ($g$ possibly not reduced).
\item $deg(g) = (1,0)$ (residual system may have double or distinct roots).
\item $deg(g) = (2,0)$ ($g$ can have a double root).
\end{enumerate}
\end{prop}

\begin{exm}\label{ex2} In the following we use capital letters to denote elements  of the various dual spaces. Note that the elements of the basis of $(\PP)^*$ that pair dually to $\{s^2u, stu, t^2u,s^2v, stv, t^2v\}$ are respectively $\{\frac{1}{2}S^2U, STU, \frac{1}{2}T^2U, \frac{1}{2}S^2V, STV, \frac{1}{2}T^2V\}$.  

Recall we have the following dual maps of linear spaces:
$$\begin{matrix} \PP \stackrel{\phi_U}{\longrightarrow} & \P(U)\stackrel{A}{\longrightarrow} \P(V) & (\PP)^*  \stackrel{\phi_U^*}{\longleftarrow} \P(U^*) \stackrel{\pi}{\longleftarrow} \P(V^*) \end{matrix}$$

In Example~\ref{ex1}, $\P(U^*) = \V(X_1,X_4-X_5) \subseteq {\P^5}^*$, $\phi^*_U(X_1)=STU$ and $\phi^*_U(X_4-X_5)=STV-\frac{1}{2}T^2V$, so there is a shared common factor $T$ of
degree $(1,0)$ and  the residual system $\{SU,\frac{1}{2}TV-SV\}$ has distinct
roots $(0:1)\times(1:0),(1:2)\times(0:1)$. Taking points $(1:0)\times(1:0)$ and 
$(1:0)\times(0:1)$ on the line $T=0$ and the points above shows that 
the forms below are in $(\phi_U^*P(U^*))^\perp={\phi_U}^{-1}(\P(U))$
\[
\begin{array}{ccc}
{\phi_U}((0:1)\times(1:0)) & = & t^2u\\
{\phi_U}((1:2)\times(0:1)) & = & (s+2t)^2v\\
{\phi_U}((1:0)\times(1:0)) & = & s^2u\\
{\phi_U}((1:0)\times(0:1)) & = & s^2v
\end{array}
\]
and in terms of our chosen basis the corresponding matrix $A$ is
\[
\left[\begin{matrix} 
0 & 0 & 1 & 0 & 0 & 0 \\ 
0 & 0 & 0 & 1 & 4 & 4 \\ 
1 & 0 & 0 & 0 & 0 & 0 \\ 
0 & 0 & 0 & 1 & 0 & 0 \\
\end{matrix}\right],
\]
whose rows span the expected linear space $U$ with basis $\langle x_0,x_2,x_3,x_4+x_5 \rangle$.
\end{exm}
\subsection{Connecting $\phi_U^*(L_1) \cap \phi_U^*(L_2)$ to syzygies}
There is a pleasant relation between Proposition~\ref{GL} and bigraded
commutative algebra. This is hinted at by the result of \cite{cds} relating the minimal free resolution of $I_W$ to $\P(W) \cap \Sigma_{2,1}$.
\begin{thm}\label{GPsyz} If $\phi_U^*(L_1) \cap \phi_U^*(L_2)$ is infinite, then:
\begin{enumerate}
\item If $deg(g) = (0,1)$ then $U$ is not basepoint  free.
\item If $deg(g) = (1,1)$ then 
\begin{enumerate}
\item if $g$ is reducible then $U$ is not basepoint  free.
\item if $g$ is irreducible then $I_U$ is of Type 3.
\end{enumerate}
\item If $deg(g) = (1,0)$ then $I_U$ is of Type 5. Furthermore
\begin{enumerate}
\item The residual scheme is reduced iff $I_U$ is of Type 5a.
\item The residual scheme has a double root iff $I_U$ is of Type 5b.

\end{enumerate}
\item If $deg(g) = (2,0)$ then $I_U$ is of Type 6. 
\end{enumerate}
\end{thm}
\begin{proof}
We use the notational conventions of \ref{ex2}. If $deg(g) = (0,1)$ then we may assume after a change of coordinates
that $\phi_U^*(L_1)= PU$ and $\phi_U^*(L_2)= QU$, with $P,Q$ of bidegree
$(2,0)$ and independent. In particular, denoting by $p,q\in R$ the dual elements of $P,Q$ a basis for $R_{2,1}$ is 
$\{pu,qu,ru, pv,qv,rv\}$. Hence $U = \spn \{ru, pv,qv,rv\}$, so
if $r = l_1(s,t)l_2(s,t)$ then $\langle v, l_1(s,t) \rangle$ is
an associated prime of $I_U$ and  $U$ is not basepoint  free.

Next, suppose $deg(g) = (1,1)$. If $g$ factors, then after 
a change of variable we may assume $g=SU$ and $U^{\perp} = \spn \{S^2U,STU \}$.
This implies that $\langle v, t \rangle$ is an associated prime of
$I_U$, so $U$ is not basepoint  free. If $g$ is irreducible, then
\[
g=a_0SU+a_1SV+a_2TU+a_3TV,\mbox{ with }a_0a_3-a_1a_2 \ne 0.
\]
Since $U^\perp = \spn \{ gS, gT \}$, $U$ is the kernel of 
\[
\left[\begin{matrix} 
a_0 & a_2 & 0 & a_1 & a_3 & 0 \\ 
0 & a_0 & a_2 & 0 & a_1 & a_3 \\
\end{matrix}\right],
\]
so $U$ contains the columns of 
\[
\left[\begin{matrix} 
a_3 & 0\\
a_1 & a_3 \\
0 & a_1 \\
-a_2 & 0 \\
-a_0 & -a_2\\
0 &  -a_0 \\
\end{matrix}\right].
\]
In particular, 
\[
\begin{array}{ccc}
a_3s^2u+a_1stu-a_2s^2v-a_0stv &= &s(a_3su+a_1tu-a_2sv-a_0tv) = sp\\
a_3stu+a_1t^2u-a_2stv-a_0t^2v &= &t(a_3su+a_1tu-a_2sv-a_0tv) = tp\\
\end{array}
\]
are both in $I_U$, yielding a linear syzygy of bidegree $(1,0)$.
Since $p$ is irreducible, the result follows from Proposition~\ref{LS3-10}. 
The proofs for the remaining two cases are similar and omitted.
\end{proof}
There is an analog of Proposition~\ref{GL} when the intersection
of the pullbacks is finite and  a corresponding connection to the minimal
free resolutions, which we leave for the interested reader.\newline
\noindent{\bf Concluding remarks} Our work raises a number of questions:
\begin{enumerate}
\item How much generalizes to other line 
bundles $\OPP(a,b)$ on $\PP$? We are at work extending 
the results of \S 7 to a more general setting.
\item What can be said about the minimal free resolution 
if $I_U$ has basepoints?
\item Is there a direct connection between embedded primes and the 
implicit equation?
\item If $U \subseteq H^0(\mathcal{O}_X(D))$ is four dimensional and
has base locus of dimension at most zero and $X$ is a toric surface, then 
the results of \cite{bdd} give a bound on the degree $\mu$ needed to 
determine the implicit equation. What can be said about the syzygies
in this case? 
\item More generally, what can be said about the multigraded free 
resolution of $I_U$, when $I_U$ is graded by $\Pic(X)$? 
\end{enumerate}

\noindent{\bf Acknowledgments} Evidence for this work was provided
by many computations done using {\tt Macaulay2}, by Dan
Grayson and Mike Stillman. {\tt Macaulay2} is freely available at 
\begin{verbatim}
http://www.math.uiuc.edu/Macaulay2/
\end{verbatim}
and scripts to perform the computations are available at 
\begin{verbatim}
http://www.math.uiuc.edu/~schenck/O21script
\end{verbatim}
We thank Nicol\'as Botbol, Marc Chardin and Claudia Polini 
for useful conversations, and an anonymous referee for a 
careful reading of the paper.
\bibliographystyle{amsplain}

\end{document}